\documentclass[reqno,a4paper,12pt]{amsart}
\usepackage{amsmath}
\usepackage{graphicx}
\usepackage[english]{babel}
\usepackage[usenames]{color}
\usepackage{dsfont}
\usepackage{subfigure}

\usepackage{ifpdf}
\ifpdf
\usepackage{hyperref}
\else
\usepackage[hypertex]{hyperref}
\hypersetup{backref, colorlinks=true, bookmarks}
\fi

\usepackage{times}
\usepackage[normalem]{ulem}

\newtheorem{theorem}{Theorem}[section]
\newtheorem{lemma}[theorem]{Lemma}
\newtheorem{definition}[theorem]{Definition}
\newtheorem{corollary}[theorem]{Corollary}
\theoremstyle{remark}
\newtheorem*{remark}{Remark}

\numberwithin{equation}{section}

\newcommand{\R}{\mathbb{R}}
\newcommand{\C}{\mathbb{C}}

\newcommand{\Z}{\mathbb{Z}}
\newcommand{\N}{\mathbb{N}}

\newcommand{\Pb}[1]{\mathbb{P}\left[#1\right]}
\newcommand{\E}[1]{\mathbb{E}\left[#1\right]}

\newcommand{\ET}[1]{\mathbb{E}_{\Theta}\left[#1\right]}
\newcommand{\PT}[1]{\mathbb{P}_{\Theta}\left[#1\right]}

\newcommand{\eps}{\epsilon}

\DeclareMathOperator{\res}{res}

\renewcommand{\Re}{\mathrm{Re}}
\renewcommand{\Im}{\mathrm{Im}}

\newcommand{\floor}[1]{\lfloor #1 \rfloor}

\newcommand{\Sn}{\mathfrak{S}_n}

\begin{document}
\title[Random permutations with logarithmic cycle weights]{Random permutations with logarithmic cycle weights}
\date{\today}

\author[N. Robles]{Nicolas Robles}
\address{Department of Mathematics, University of Illinois, 1409 West Green Street, Urbana, IL 61801, United States \textnormal{and} Wolfram Research Inc, 100 Trade Center Dr, Champaign, IL 61820, United States}
\email{nirobles@illinois.edu}
\email{nicolasr@wolfram.com}

\author[D. Zeindler]{Dirk Zeindler}
\address{Department of Mathematics and Statistics, Lancaster University, Fylde College, Bailrigg, Lancaster LA1 4YF, United Kingdom}
\email{d.zeindler@lancaster.ac.uk}

\begin{abstract}
We consider random permutations on $\Sn$ with logarithmic growing cycles weights 
and study asymptotic behavior as the length $n$ tends to infinity. 
We show that the cycle count process converges to a vector of independent Poisson variables and also compute the total variation distance between both processes. 
Next, we prove a central limit theorem for the total number of cycles.
Furthermore we establish 
a shape theorem and a functional central limit theorem for the Young diagrams associated to random  permutations under this measure.
We prove these results using tools from complex analysis and combinatorics. 
In particular we have to apply the method of singularity analysis to generating functions 
of the form $\exp\left( (-\log(1-z))^{k+1}  \right)$ with $k\geq 1$, which have not yet been studied in the literature. 
\end{abstract}

\keywords{random permutations, cycle counts, total variation distance, total number of cycles, singularity analysis, limit shape, functional central limit theorem, Tauberian theorem}
\subjclass{60F17, 60F05, 60C05, 40E05}

\maketitle

%---------------------------------------------------------------------------
%\tableofcontents

\section{Introduction}
\label{sec_intro}

Let $\Sn$ be the symmetric group of all permutations on elements
$1,\dots,n$. For any permutation $\sigma\in \Sn$, denote by
$C_m=C_m(\sigma)$ the \emph{cycle counts}, that is, the number of
cycles of length $m=1,\dots,n$ in the cycle decomposition
of~$\sigma$; clearly
\begin{equation}\label{eq:sumC}
C_m\ge0 \quad(m\ge1),\qquad \sum_{m=1}^n m\, C_m=n.
\end{equation}
Here we study random permutations with respect to the following probability measure
\begin{definition}
\label{def:Pb_measure}
Let $\Theta = \left(\theta_m  \right)_{m\geq1}$ be given, with $\theta_m\geq0$ for every $m\geq 1$.
We define for $\sigma\in\Sn$ the weighted measures on $\Sn$ as
\begin{align}
  \PT{\sigma}
  :=
  \frac{1}{h_n n!} \prod_{m=1}^n \theta_m^{C_m}
  \label{eq:PTheta_with_partition}
\end{align}
with $h_n = h_n(\Theta)$ a normalization constant and $h_0:=1$.
\end{definition}
This measure has received a lot of attention in recent years and has been studied by many authors. 
An overview can be found  in \cite{ErUe11}.
Classical cases of $\mathbb{P}_\Theta$ are the uniform measure ($\theta_m \equiv 1$) and the Ewens measure ($\theta_m \equiv \theta$). 
The uniform measure is well studied and has a long history
(see e.g. the first chapter of \cite{ABT02} for a detailed account with references).
 The Ewens measure originally appeared in population genetics, 
see \cite{Ew72}, but has also various applications through its connection with Kingman's coalescent process, see \cite{Ho87}.

The motivation to study the measure $\mathbb{P}_\Theta$ has its origins in  mathematical physics.
Explicitly, it occurred in the context of the Feynman-Kac representation of the dilute Bose gas and 
it has been proposed in connection with the study of the Bose-Einstein condensation (see e.g. \cite{BeUeVe11} and \cite{ErUe11}). 
An important question in this context, which is
also interesting on its own right, is the possible emergence of cycles with a cycle length with order of magnitude $n$ as $n\to\infty$.
It is clear that the asymptotic behaviour of the measure $\mathbb{P}_\Theta$ as $n\to\infty$  strongly depends on the sequence $\Theta = \left(\theta_m  \right)_{m\geq1}$.
In the current literature, only the cases $\theta_m\approx \vartheta$ and $\theta_m\sim m^\gamma$ with $\gamma>0$ are well studied.
It is known that in the case $\theta_m\approx \vartheta$ there are cycles of order $n$ in the limit 
and that the longest cycles follow a Poisson-Dirichlet distribution, see \cite{Ki77, ShVe77, ErUe11, BeUe10}.
On the other hand, it was shown in \cite{CiZe13, ErUe11} that in the case $\theta_m\sim m^\gamma$  most cycles have a cycle length of order $n^{\frac{1}{1+\gamma}}$ 
and thus are no cycles of order $n$ in the limit. Furthermore, it was established in \cite{CiZe13} that the Young diagrams associated to random permutations
converges in this situation to a limit shape.
In this paper, we consider the cycle weights of the form 
\begin{align}
 \theta_m =  \log^k m \text{ for } m\in\N \text{ and some }k\in\N.
 \label{eq:def_theta}
\end{align}
We use in fact sightly more general weights than in \eqref{eq:def_theta} and our exact assumptions are given in Section~\ref{sec:assumption_theta}.
Weights of the form \eqref{eq:def_theta} have not been studied in the literature and our motivation to consider these weights is the following question. 
Are there any cycles of order $n$ in the limit if one is considering slowly growing cycles weights $\theta_m$ as $m\to\infty$?
We show in this paper that the length of typical cycle under this measure has the order of magnitude $n/\log^k n$ (see Theorem~\ref{thm:L1L2} and Theorem~\ref{thm:limit_shape_saddle}) 
and thus there are no cycles with lengths of order $n$.
Also, we show the following. For each $b\in\N$ fix, we have as $n\to\infty$
\begin{align}
  \left(C_1, C_2,  \ldots,C_b \right)
  \stackrel{d}{\to}
  \left(Y_1,\ldots,Y_b\right)
\end{align}
with $Y_1,\cdots,Y_b$ independent Poisson distributed random variables with $\E{Y_m} = \frac{\theta_m}{m}$, see Theorem~\ref{thm:limit_theorem_cycles}.
Further, we compute the total variation distance between both processes and show that this is tending to $0$ for $b=o(n^c)$ for some $c\in(0,1)$, see Theorem~\ref{thm:DTV}.
Moreover, we prove a central limit theorem for the total number of cycles, see Theorem~\ref{thm:limit_theorem_total_cycles}, and show that
a typical permutation consists in average of $\frac{\log^{k+1}(n)}{k+1}$ disjoint cycles.
Finally, we establish in Section~\ref{sec:limit_shape} a shape theorem and a functional central limit theorem for the Young diagrams associated to random  permutations. 

We prove these results using tools from complex analysis and combinatorics. 
For this, we have in particular to compute the asymptotic behaviour of
\begin{align}
 [z^n]\left[\exp\left( (-\log(1-z))^{k+1}  \right)\right]
 \label{eq:gen_intro}
\end{align}
as $n\to\infty$. As far as we are aware, this has not yet been studied in the literature and 
we compute \eqref{eq:gen_intro} with a modified version of the saddle point method, see Theorem~\ref{thm:aux_asypmtotic}.

\section{Generating functions and asymptotic theorems}
\label{sec:gen_and_asympt}
We recall in Section~\ref{sec:gfs} some basic facts about $\Sn$ and generating functions. 
This includes P\'olya's Enumeration Theorem, which is a useful tool to perform averages on the symmetric group.
In Section~\ref{sec:asymp_theorems}, we determine some analytic properties of the generating functions occurring in this paper and 
establish a result, see Theorem~\ref{thm:aux_asypmtotic}, which enables us to compute the asymptotic behaviour of the expression in \eqref{eq:gen_intro}.
\subsection{Generating functions}
\label{sec:gfs}
We use standard notation $\Z$ and $\N$ for the sets of integer and
natural numbers, respectively, and also denote
$\N_0:=\{m\in\Z:\,m\ge0\}=\{0\}\cup\N$.

For a sequence of complex numbers $(a_m)_{m\ge 0}$, its (ordinary)
generating function is defined as the formal power series
\begin{align}
\label{eq:G}
g(t)
:= 
\sum_{m=0}^\infty a_m t^m.
\end{align}
As usual \cite[\S{}I.1, p.\,19]{FlSe09}, we define the
\emph{extraction symbol} $[t^m]\, g(t):= a_m$, that is, as the
coefficient of $t^m$ in the power series expansion \eqref{eq:G}
of~$g(t)$.

The following simple lemma known as \emph{Pringsheim's Theorem}
(see, e.g., \cite[Theorem IV.6, p.\;240]{FlSe09}) is important in
asymptotic enumeration where generating functions with non-negative
coefficients are usually involved.
\begin{lemma}\label{lm:Pringsheim}
Assume that $a_m\ge0$ for all $m\ge0$, and let the series expansion
\eqref{eq:G} have a finite radius of convergence $R$. Then the point
$t = R$ is a singularity of the function $g(t)$.
\end{lemma}

A special generating function constructed with the coefficients
$(\theta_m)$ plays a crucial role in
this paper, i.e.
\begin{equation}
\label{eq:def_g_theta}
  g_\Theta(t):=
  \sum_{m=1}^\infty \frac{\theta_m}{m}\, t^m.
\end{equation}
Indeed, we will see, the asymptotic behaviour of the measure $\mathbb{P}_\Theta$ is
determined by the analytic properties of the function $g_\Theta(z)$.

Recall that the cycle counts $C_m=C_m(\sigma)$ are defined as the
number of cycles of length $m\in\N$ in the cycle decomposition of
permutation $\sigma\in\Sn$ (see the Introduction). The next
well-known identity is a special case of the general \emph{P\'olya's
Enumeration Theorem} \cite[\S16, p.\,17]{Po37} and the proof can be found
for instance in \cite[p.~5]{Mac95}).
\begin{lemma}
\label{lem:cycle_index_theorem}
Let $(a_m)_{m\in\N}$ be a sequence of \textup{(}real or
complex\textup{)} numbers. Then there is the following
\textup{(}formal\textup{)} power series expansion
\begin{equation}\label{eq:symm_fkt}
\exp\left(\sum_{m=1}^{\infty}\frac{a_m t^m}{m}\right)
=\sum_{n=0}^\infty\frac{t^n}{n!}\sum_{\sigma\in\Sn}\prod_{m=1}^{n}
a_{m}^{C_m},
\end{equation}
where $C_m=C_m(\sigma)$ are the cycle counts. If either of the
series in \eqref{eq:symm_fkt} is absolutely convergent then so is
the other one.
\end{lemma}
We get immediately that 
\begin{corollary}
\label{cor:Hn_generatingN}
Let $h_n$ be the normalisation constant in Definition~\ref{def:Pb_measure}. 
 We then have as formal power series in $t$
 \begin{align}
  \sum_{n=0}^\infty h_n t^n
  =
  \exp\bigl(g_\Theta(t)\bigr).
  \label{eq:Hn_generatingN}
 \end{align}
\end{corollary}

\subsection{Asymptotic theorems for generating function}
\label{sec:asymp_theorems}

In this section, we develop complex-analytic tools for computing the asymptotics of the coefficient $h_{n}$ in the power series
expansion of $  \exp\bigl(g_\Theta(t)\bigr)$ (see \eqref{eq:Hn_generatingN})
for the cycle weights $\theta_m$ in \eqref{eq:def_theta}. More generally, it is useful to consider expansions of the function
$\exp\bigl(vg_\Theta(t)\bigr)$, with some parameter $v>0$. 
We will see that the case $v=1$ is of primary importance, but we will need at some certain also
the behavior for $v\approx 1$ to deduce some limit theorems. 

Note that the function $g_\Theta(t)$ has radius of convergence $1$. A big part of our argumentation is based on the saddle-point method. 
For this we require the asymptotic behavior as $t\to 1$.
Note that the function 
 $$
 g_\Theta(t)
 =
 \sum_{m=1}^\infty \frac{\log^k m}{m} t^m
 $$
is a special case of the polylogarithm, see \cite[\S VI.8]{FlSe09} and \cite{Fl99} as well as \cite{dr01} for uses of the polylog in polynomial partitions.
We thus summarize here only the properties we need and give only a sketch of the proofs.
For a detailed proof, we refer to \cite{FlSe09}.
\begin{lemma}
\label{lem:asymptotic_g}
  Let $\theta_m$ in \eqref{eq:def_theta}. We then have 
 \begin{align}
  g_\Theta(t)
  =
  \sum_{m=1}^\infty \frac{\log^k m}{m} t^m
 \end{align}
and the function  $g_\Theta(t)$ can be analytically continued to $\C\setminus [1,\infty]$. 
Further, there exists a polynomial $P$ with 
\begin{align}
 P(r)= \frac{r^{k+1}}{k+1}+ \sum_{j=0}^{k} c_j r^j
\end{align}
with $c_j\in\R$ for $0\leq j\leq k$
such that  
 \begin{align}
  g_\Theta(e^{-w})
    =
  P\bigl(-\log(w)\bigr)
  +
  O(w)
  \label{eq:g_theta_to_1}
 \end{align}
for $w\to 0$ with $\arg(w)\leq  \pi- \epsilon$ and $\epsilon>0$ arbitrary.
\end{lemma}
Equation \eqref{eq:g_theta_to_1} is related to \eqref{eq:gen_intro} by inserting $w=-\log(z)$ and then expanding.
 Indeed, we have as $z\to1$ with $|z|<1$ that
 \begin{align*}
  g_\Theta(z) 
  &=  
  P\bigl(-\log(-\log z)\bigr)
  +
  O(z-1)
  =
   P\Big(-\log\Big(-\log\big(1 +(z-1)\bigr)\Bigr)\Big)
  +
  O(z-1)\\
  &=
  P\Big(-\log\Big(-(z-1)+ O\big((z-1)^2\big) \Bigr)\Big)
  +
  O(z-1)\\
    &=
  P\Big(-\log\big(-(z-1)\big)+ O(z-1) \Big)
  +
  O(z-1)\\
   &=
  P\big(-\log(1-z)\big)
  +
  O\big((z-1)^{1/2}\big).
 \end{align*}
Inserting this computation into the generating function of $h_n$ in \eqref{eq:Hn_generatingN}, we indeed get \eqref{eq:gen_intro}.
However, we will work with the expression $g_\Theta(e^{-w})$ instead $g_\Theta(z)$ as this is more convenient in our computations.

\begin{proof}[Sketch of proof]
The function $g_\Theta(t)$ has clearly radius of convergence $1$ and is thus analytic for $|t|<1$.
For the analytic continuation, one use Lindel\"ofs integral representation of the polylogarithm, namely
\begin{align}
  g_\Theta(-t)
  =
  \frac{-1}{2\pi i} \int_{1/2-i\infty}^{1/2+i\infty}  \frac{\log^k (s)}{s} \frac{t^s \pi}{\sin (\pi s)} \,ds.
\end{align}
It is now easy to see that this integral is absolutely convergent for $t\in\C\setminus [0,\infty]$ and that it defines 
in $\C\setminus [0,\infty]$ an analytic function. Combining this with the fact that $g_\Theta(t)$ has radius of convergence $1$, proves the first part of the lemma.

To compute the asymptotic behaviour of $g_\Theta(e^{-w})$ as $w\to 0$, we use the Mellin transform, see for instance \cite[\S B.7]{FlSe09}.
Applying some elementary properties of the Mellin transform, we get immediately 
\begin{align}
 g^{*}_\Theta(s) := \int_0^\infty g_\Theta(e^{-w}) w^{s-1} \ dw = (-1)^k\zeta^{(k)}(s+1) \Gamma(s),
\end{align}
where $\zeta^{(k)}(s)$ is the $k$'th derivative of the Riemann zeta function and $\Gamma$ is the Gamma function. 
%For applications of $\zeta^{(k)}$, see \cite{pr01, prrz01}.
Using the inverse Mellin transform, we obtain 
 \begin{align}
  g_\Theta(e^{-w})
  =
  \int_{1/2-i\infty }^{1/2+i\infty } (-1)^k \zeta^{(k)}(s+1) \Gamma(s) w^{-s}\, ds.
 \end{align}
We now shift the contour of integration to $\Re(s) =-3/2$. By doing this, we pick up poles at $s=0$ and at $s=-1$ so that
  \begin{align}
  g_\Theta(e^{-w})
  =&
  \int_{-3/2-i\infty }^{-3/2+i\infty } (-1)^k \zeta^{(k)}(s+1) \Gamma(s) w^{-s}\, ds\label{eq:asymptotic_g1}\\
  &+  \res_{s=0}\left( (-1)^k \zeta^{(k)}(s+1) \Gamma(s) w^{-s}\right)
  +  \res_{s=-1}\left( (-1)^k \zeta^{(k)}(s+1) \Gamma(s) w^{-s}\right).
%   \\
%   &= (-1)^{k+1}\frac{\log^{k+1}(w)}{k+1} +O(\log^{k+1}(w))
%   +
%   \int_{-1/2-i\infty }^{-1/2+i\infty } (-1)^k \zeta^{(k)}(s+1) \Gamma(s) w^{-s}\, ds
\nonumber
 \end{align}
We consider the Laurent expansion of $ (-1)^k \zeta^{(k)}(s+1) \Gamma(s)$ around $s=0$ and get
\begin{align}
 (-1)^k \zeta^{(k)}(s+1) \Gamma(s) = k!s^{-2-k} + \sum_{j=0}^{k} d_j s^{-j-1} +O(1),
\end{align}
for some  $d_j\in\R$, $0\leq j\leq k$. Note that this Laurent expansion is independent of $w$.  Using the Taylor expansion of $w^{-s}=e^{-s\log w}$ around $s=0$ then gives
\begin{align}
 \res_{s=0}\left( (-1)^k \zeta^{(k)}(s+1) \Gamma(s) w^{-s}\right)
 &=
 (-1)^{k+1}\frac{\log^{k+1}(w)}{k+1}
 +
 \sum_{j=0}^{k} d_j (-1)^j \frac{\log^{j}(w)}{j!}\\
&=
 \frac{1}{k+1}(-\log(w))^{k+1}
 +
 \sum_{j=0}^{k} c_j (-\log(w))^{j}\\
 &=
  P\bigl(-\log(w)\bigr)
\end{align}
with $c_j= d_j/j!$. Thus the residue at $s=0$ has the form we are looking for. Since $\Gamma(s)$ has a simple pole with residue $-1$ at $s=-1$, we get that 
\begin{align}
 \res_{s=-1}\left( (-1)^k \zeta^{(k)}(s+1) \Gamma(s) w^{-s}\right)
 =(-1)^{k+1} \zeta(0)w.
 \end{align}
The integral in \eqref{eq:asymptotic_g1} is well defined for all $w$ with $\arg(w)\leq  \pi/2- \epsilon$ since $|\Gamma(\sigma+it)| = O(t^2e^{-\frac{\pi}{2}t})$ for $|t|\to\infty$ and $\sigma>-2$.
A direct estimate then shows that this integral is of order $O(w^{3/2}) $. 
This shows that the above expansion in \eqref{eq:g_theta_to_1} is valid for  $\arg(w)\leq  \pi/2- \epsilon$.
To complete the proof, it remains to show that this expansion is also valid for $|\arg(w)|\leq  \pi- \epsilon$. 
We omit this proof as it follows the same lines as in the proof of \cite[Lemma~3]{Fl99}.
\end{proof}

\begin{remark}
One can easily relate the coefficients $c_j$ in Lemma~\ref{lem:asymptotic_g} to the Laurent expansion of $\Gamma(s)$ around $s=0$. 
However, for our purpose it is enough to know that the $c_j$ are real numbers. 
Further, one can obtain with the above argumentation a complete asymptotic expansion of $ g_\Theta(e^{-w})$ 
and this asymptotic expansion is valid for $|\arg(w)|\leq  \pi- \epsilon$. However, we do not need it here and thus will not prove it. 
Details can be found for instance in \cite[\S B.7]{FlSe09} and \cite{Fl99}.
\end{remark}

\begin{theorem}
\label{thm:aux_asypmtotic}

Let $g_\Theta(t)$ be as in \eqref{eq:def_g_theta}. Suppose $g_\Theta(t)$ has the radius of convergence $1$ and that $g_\Theta(t)$ is continuous in the punctured disc $\{|t|\leq 1$, $t\neq 1\}$. 
Suppose further there exists a polynomial $P$ with 
\begin{align}
 P(r)= \frac{r^{k+1}}{k+1}+ \sum_{j=0}^{k} c_j r^j
\end{align}
with $k\geq 1$   
such that  
 \begin{align}
  g_\Theta(e^{-w})
    =
  P\bigl(-\log(w)\bigr)
  +
  O(w)
%   =
%   (-1)^{k+1}\frac{\log^{k+1}(w)}{k+1} 
%   +
%   \sum_{j=0}^{k} (-1)^j c_j \log^{j}(w)
%   +
%   O(w)
  \label{eq:g_theta_to_2}
 \end{align}
for $w\to 0$ with $\arg(w)\leq  \pi/2$.
We then have for $v >0$

\begin{align}
[t^n]\left(\exp\bigl(vg_\Theta(t)\bigr)\right)
= 
\frac{\exp\left(vP(r)+ ne^{-r}\right) }{e^{r}\sqrt{2\pi vP''(r)+2\pi ne^{-r}}} \left(1+ O(\log^{-k/2}(n)) \right),
\label{eq:thm:aux_asypmtotic1}
\end{align}
where $r$ is a solution of the equation
\begin{align}
 vP'(r)= ne^{-r}.
 \label{eq:thm:aux_asypmtotic_saddle_eq}
\end{align}
Furthermore, the error term in \eqref{eq:thm:aux_asypmtotic1} is uniform in $v$ for $v \in [v_1 , v_2 ]$, where $v_1$, $v_2$ are arbitrary, but fixed constants with $1\leq v_1<v_2<\infty$.	
\end{theorem}
We have $P'(r)\sim r^k$ as $r\to\infty$ and thus \eqref{eq:thm:aux_asypmtotic_saddle_eq} has a solution for $n$ large.
Note that the solution $r$ is unique if $c_j\geq 0$ for all $j$.
This does not have to be the case if some of the $c_j$ are negative.  
However, a straight forward computation shows that all solution fulfills the same asymptotic expansion
\begin{align}
r = \log(n/v) - k\log\log(n/v) +O(1) \ \text{ as }n \to\infty.
\label{eq:saddle_solution_asympt}
\end{align}
From this, we immediately get
\begin{align}
 P(r) &= \frac{\log^{k+1}(n) }{k+1}\bigl(1+ O\bigl(\log^{-1}(n)\bigr)\bigr), \
  P''(r) = k\log^{k-1}(n) \bigl(1+ O\bigl(\log^{-1}(n)\bigr)\bigr)
   \label{eq:P''(r)}\\
 vP'(r) &=  ne^{-r} = v\log^{k}(n) \bigl(1+ O\bigl(\log^{-1}(n)\bigr)\bigr).
%  P''(r) &= k\log^{k-1}(n) \bigl(1+ O\bigl(\log^{-1}(n)\bigr)\bigr)
 \label{eq:P'(r)}
\end{align}
For proof of Theorem~\ref{thm:aux_asypmtotic} we will use the saddle point method. 
Unfortunately, the function $g_\Theta(t)$ is in this situation not (log-)Hayman admissible (see \cite{CiZe13} and \cite[\S VIII.5.]{FlSe09}). 
We thus cannot use the standardized saddle point method, which is described for instance in \cite[\S VIII.5.]{FlSe09}).
We therefore use a slightly modified version. Also, we need an auxiliary result.

\begin{lemma}
\label{lem:int_exp_P}
Let $C>0$ be given. Let further $Q(x)= a_dx^d+\ldots+a_0$ be a real polynomial with $a_d>0$ and $d\geq2$. We then have as $r\to\infty$
\begin{align}
 \int_{C}^r \exp\left( Q(y) \right)\, dy
 =
 \frac{1}{Q'(r)} \exp\left( Q(r) \right) + O\left( \frac{1}{(Q'(r))^2} \exp\left( Q(r) \right)   \right).
 \label{eq:int_exp_P}
\end{align}
\end{lemma}
\begin{proof}
 We first chose $\delta = \delta(r)$ with $ \delta=O(1)$ and $\delta\cdot Q'(r)/\log(r) \to \infty$. 
 We then split the integral in \eqref{eq:int_exp_P} into the integrals over $[C,\,r-\delta]$ and $[r-\delta, r]$. 

 We first consider the part over $[C,\,r-\delta]$. 
 For $r$ large enough, $Q(y)$ attains its maximum in the interval $[C,\,r-\delta]$ at the point $r-\delta$.
 Furthermore, 
 \begin{align*}
  Q(r-\delta) =  Q(r) - \delta Q'(r) + \frac{\delta^2}{2} Q''(\xi)
  \ \text{ for some }\xi\in [r-\delta, r]. 
 \end{align*}
%  for some $\xi\in [r-\delta, r]$. 
 Since $\delta$ is bounded, we have for $r$ large enough
  \begin{align*}
  Q(r-\delta) \leq Q(r) - \frac{\delta}{2} Q'(r).
 \end{align*}
 Using this and the trivial estimate, we get
 \begin{align}
 \left|\int_{C}^{r-\delta} \exp\left( Q(y) \right)\, dy \right|
 &\leq 
  r\exp\Big( Q(r-\delta) \Big)
 \leq 
 \exp\left( Q(r) + \log(r) - \frac{\delta}{2} Q'(r) \right).
\end{align}
By assumption, we have $\delta\cdot Q'(r)/\log(r) \to \infty$ and thus $\log(r) - \frac{\delta}{2} Q'(r) \leq -K \log(r)$ for $r$ large enough,
where $K$ can be chosen arbitrary large. This implies that 
 \begin{align}
\int_{C}^{r-\delta} \exp\left( Q(y) \right)\, dy 
=
 O\left(\exp\Big( Q(r) \Big) r^{-K} \right).
\end{align}
This shows that the integral over $[C,\,r-\delta]$ is of lower order. 
For the integral over $[r-\delta,r]$, we use partial integration and a similar estimate as above to obtain
\begin{align*}
 \int_{r-\delta}^r \exp\left( Q(y) \right)\, dy
 =\,&
 \int_{r-\delta}^r \frac{1}{Q'(y)} \left(Q'(y) e^{Q(y)} \right)\, dy
 =
 \left. \frac{1}{Q'(y)} e^{Q(y)} \right|_{y=r-\delta}^r 
 -
 \int_{r-\delta}^r \frac{1}{Q'(y)} e^{Q(y)} \, dy\\
 =\,&
 \frac{1}{Q'(r)} e^{Q(y)} 
 -
 \int_{r-\delta}^r \frac{1}{Q'(y)} \exp\big( Q(y) \big)\, dy
 +  
 O\left( e^{Q(y)}  r^{-K} \right)\\
 =\,&
 \frac{1}{Q'(r)} e^{Q(y)} 
 -
\frac{1}{(Q'(r))^2}  e^{Q(y)} 
 +
 \int_{r-\delta}^r \frac{1}{(Q'(y))^2} e^{Q(y)}  \, dy
 +  
 O\left(e^{ Q(r) } r^{-K} \right)\\
 =\,&
 \frac{1}{Q'(r)} e^{Q(y)} 
 +  
 O\left( \frac{1}{(Q'(r))^2}  e^{Q(y)} \right).
\end{align*}
This completes the proof.
\end{proof}

\begin{proof}[Proof of Theorem~\ref{thm:aux_asypmtotic}]
We use Cauchy's integral formula and get
\begin{align}
 I_n:=[t^n]\left(\exp\bigl(vg_\Theta(t)\bigr)\right)
 = \frac{1}{2\pi i} \oint_\gamma \exp\bigl(vg_\Theta(t)\bigr) \frac{1}{t^{n+1}}\, dt,
\end{align}
where $\gamma$ is the circle $\gamma:=\{t=e^{-1/2}e^{i\varphi},\varphi\in[-\pi,\pi]\}$.
Applying the variable substitution $t=e^{-w}$, we get 
\begin{align}
 I_n
 = 
 \frac{1}{2\pi i} \int_{\gamma'} \exp\bigl(vg_\Theta(e^{-w})\bigr) e^{nw}\, dw
 \label{eq:aux_therorem1}
\end{align}
 with $\gamma':=\{t=1/2 +is,\,s\in[-\pi,\pi]\}$.
Note that the integrand in \eqref{eq:aux_therorem1} is $2\pi i$ periodic. 
We thus can shift the contour $\gamma'$ to the contour $\gamma''= \gamma''_1\cup \gamma''_2\cup\gamma''_3$ (see Figure~\ref{fig:gamma2''}) with 
 \begin{align*}
\gamma''_1&:=\{w=(-\pi+x)i,x\in[0,\pi-e^{-r}]\},\\
\gamma''_2&:=\{w=e^{-r}e^{i\varphi},\varphi\in[-\pi/2 ,\pi/2\},\\
\gamma''_3&:=\{w=ix ,x\in[e^{-r},\pi]\},
 \end{align*}
 where $r$ is the solution of the equation \eqref{eq:thm:aux_asypmtotic_saddle_eq}.
\begin{figure}[ht!]
 \centering
 \includegraphics[width=.35\textwidth]{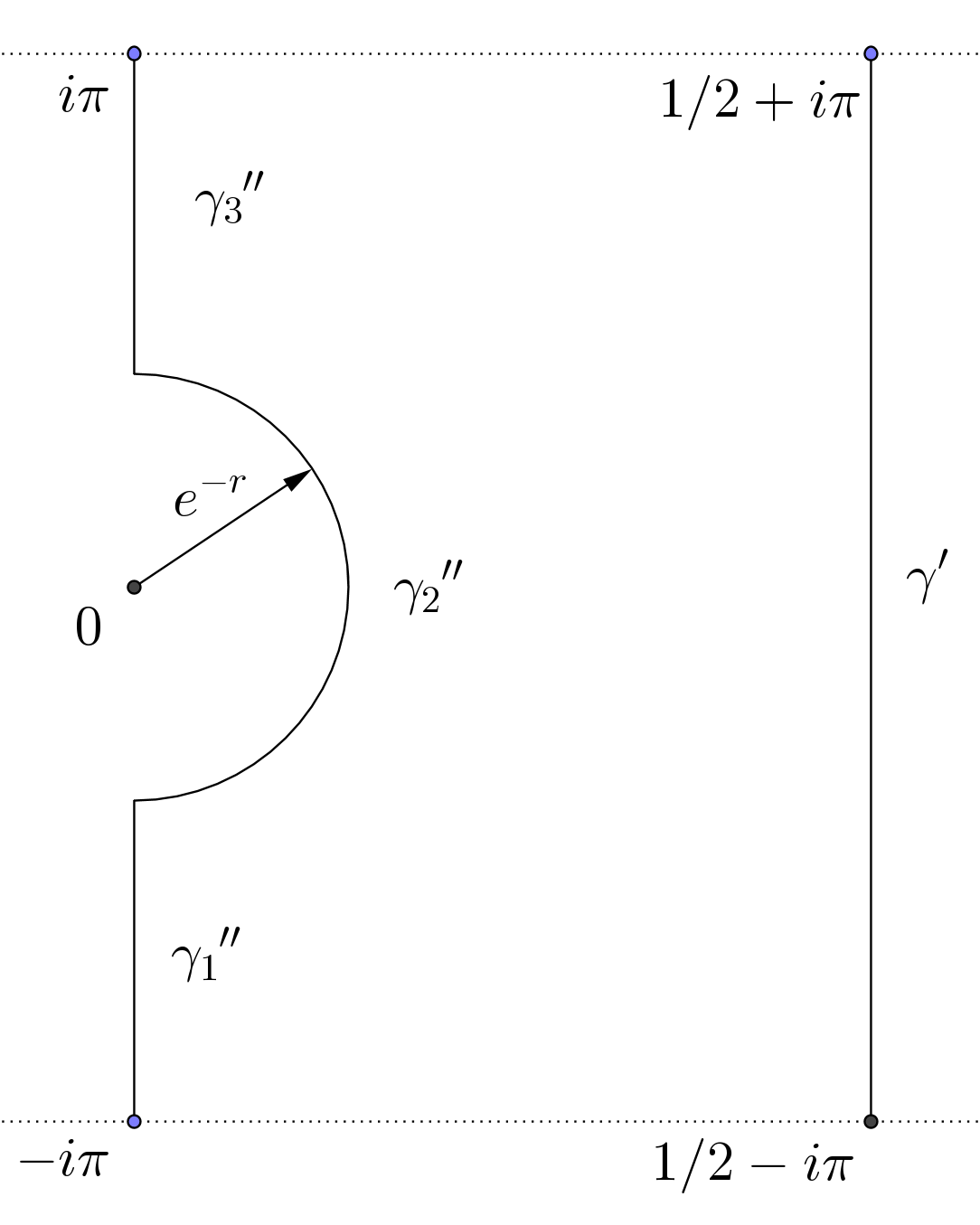}
 \caption{The contours $\gamma'$ and $\gamma''$.}
 \label{fig:gamma2''}
\end{figure}
% \comment{Picture need here????}
We thus can write $I_n = I_{n,1}+ I_{n,2}+I_{n,3}$, where $I_{n,j}$ is the integral over $\gamma''_j$.

We begin by computing $I_{n,2}$ with the saddle point method. 
We thus take first a look at the behaviour of the integrand in $I_{n,2}$ for $\varphi$ around $0$.
We use \eqref{eq:g_theta_to_2} and get 
\begin{align}
 g_\Theta\left(e^{-e^{-r}e^{i\varphi}}\right)
 =&
 P(r-i\varphi) +O(e^{-r}e^{i\varphi}).
 \end{align}
 Expanding $P(r-i\varphi)$ around $\varphi =0$ gives
 \begin{align}
%   g_\Theta\left(e^{-e^{-r}e^{i\varphi}}\right)
%   =&
  P(r-i\varphi)
  =&
  P(r)-i\varphi P'(r) -\frac{1}{2}P''(r)\varphi^2 + 
  O\left(\varphi^3 r^{k-2} \right) \ \text{ as }\varphi \to 0.
\end{align}
% 
% %
We now split the integral $I_{n,2}$ into the regions $[-\delta,\delta]$ and $[-\pi/2, \pi/2]\setminus [-\delta,\delta]$ for some $\delta>0$ small determined below.
We first take a look at the integral over $[-\delta,\delta]$. 
With \eqref{eq:g_theta_to_2} we get
\begin{align*}
I_{n,2,\delta}
:=
&\frac{e^{-r}}{2\pi} \int_{-\delta}^{\delta} \exp\left(vg\left(e^{-e^{-r}e^{i\varphi}}\right) +ne^{-r}e^{i\varphi}+i\varphi      \right) \, d\varphi \\
=&
\frac{1}{2\pi e^{r}} \int_{-\delta}^{\delta} \exp\bigl(v P(r-i\varphi) +ne^{-r}e^{i\varphi}+i\varphi+O(ve^{-r}e^{i\varphi})      \bigr) \, d\varphi \\
=
 & \frac{1}{2\pi e^{r}} \int_{-\delta}^{\delta} \exp\left(v\left(P(r)-i\varphi P'(r) -\frac{1}{2}P''(r)\varphi^2 + 
 O\left(\varphi^3 r^{k-2}\right)\right)   \right) \\
&\qquad \times \exp\left(ne^{-r}e^{i\varphi}+i\varphi +O(e^{-r})     \right) \, d\varphi.
\end{align*}
Expanding $ne^{-r}e^{i\varphi}$ around $\varphi=0$ and using that we have $vP'(r)= ne^{-r}$ by the definition of $r$ in \eqref{eq:thm:aux_asypmtotic_saddle_eq}, we obtain
\begin{align*}
I_{n,2,\delta}
=&\,
\frac{\exp\bigl(vP(r)+ne^{-r}\bigr) }{2\pi e^r} \\
&\times\int_{-\delta}^{\delta} \exp\left({ -\frac{1}{2}(vP''(r)+ne^{-r})\varphi^2}\right)
\exp\left({ i\varphi +O\left(\varphi^3 (r^{k-2}+ne^{-r})\right) +O(e^{-r})}\right)\, d\varphi.  
\end{align*}
We know from \eqref{eq:saddle_solution_asympt}, \eqref{eq:P''(r)} and \eqref{eq:P'(r)} that 
\begin{align}
 r \sim \log(n), \   P''(r)\sim k\log^{k-1}(n) \ \text{ and } \ ne^{-r} \sim v\log^{k}(n)  \ \text{ as } n\to\infty.
\end{align}
Thus $ne^{-r}$ is dominating in the coefficients of $\varphi^2$ and $\varphi^3$ in the above expression for $I_{n,2,\delta}$.
We now define $\delta:=\delta(n,v) = (ne^{-r})^{-5/12}$. Thus $\delta\to 0$ and 
\begin{align}
 \delta^2 (vP''(r)+ne^{-r}) \to \infty \ \text{ and }\ \delta^3(r^{k-2}+ne^{-r}) \to 0.
\end{align}
We therefore get
\begin{align*}
I_{n,2,\delta}
&=
\frac{\exp\left(vP(r)+ne^{-r}\right) }{2\pi e^r} 
\int_{-\delta}^{\delta} e^{ -\frac{1}{2}(vP''(r)+ne^{-r})\varphi^2}
(1-i\varphi +O\left(\varphi^2+ \varphi^3 ne^{-r}+e^{-r})\right)\, d\varphi.  
% \exp\left( -i\varphi +O\left(\varphi^3 (r^{k-2}+ne^{-r})\right) +O(e^{-r})  \right)   \\
% \times \exp\left(ne^{-r}e^{i\varphi}-i\varphi +O(e^{-r}e^{i\varphi})     \right) \, d\varphi
\end{align*}
For notational convince, we write $b:=vP''(r)+ne^{-r}$.
The function $\varphi\, e^{ -\frac{ b}{2}\varphi^2}$ is odd and thus we can remove the $i\varphi$ in the last equation. 
Using the variable substitution $x^2 =  b\varphi^2$, we get
 \begin{align}
&\int_{-\delta}^{\delta} e^{ - \frac{b}{2}\varphi^2}\left(1+ O(\varphi^2)+O\left( ne^{-r}\varphi^3\right)+O\left(e^{-r}\right)\right) \, d\varphi \nonumber\\
 =&
\frac{1}{\sqrt{b}}\int_{-\delta\sqrt{b}}^{\delta\sqrt{b}} e^{ - \frac{1}{2}x^2}\left(1+O(b^{-1}x^2)+O\left( ne^{-r}b^{-3/2}x^3\right)+O\left(e^{-r}\right)\right) \, dx \nonumber\\
=&
% \frac{1}{\sqrt{b}}\int_{-\delta\sqrt{b}}^{\delta\sqrt{b}} e^{ - \frac{1}{2}x^2}\, dx \left(1+O(b^{-1})+O\left( ne^{-r}b^{-3/2}\right)+O\left(e^{-r}\right)\right) \, dx\\
% =&
\frac{1}{\sqrt{b}}\left(\int_{-\infty}^{\infty} e^{ - \frac{1}{2}x^2}\, dx +O(e^{-\delta\sqrt{b}})\right) \left(1+O(b^{-1})+O\left( ne^{-r}b^{-3/2}\right)+O\left(e^{-r}\right)\right)\nonumber \\
=&
\frac{\sqrt{2\pi}}{\sqrt{b}} \left(1+ O(\log^{-k/2}(n)) \right).
\label{eq:thm_aux_main_term_saddle}
 \end{align}
We thus obtain
\begin{align}
I_{n,2,\delta}
&=
  \frac{\exp\left(vP(r)+ne^{-r}\right) }{e^r\sqrt{2\pi (vP''(r)+ne^{-r})}} \left(1+ O(\log^{-k/2}(n)) \right).
%&=
%\frac{\exp\left(vP(r)+ne^{-r}\right) }{\sqrt{2\pi \log^k(n)}} \left(1+ O(\log^{-k/2}(n)) \right)
\end{align}
We now show that remaining parts of $I_{n,2}$ and $I_{n,1}$, $I_{n,3}$ are all of lower order. We denote by $I^c_{n,2}$ the remaining part of the $I_{n,2}$, i.e. $I^c_{n,2}=I_{n,2}-I_{n,2,\delta}$.
For this, we use the inequalities 
\begin{align}
 \cos(\varphi)&\leq 1-\varphi^2/12 \ \text{ for }|\varphi|\leq \pi \ \text{ and }\\
%  \Re\left(\frac{ (r-i\varphi)^{k+1}}{k+1}+ \sum_{j=0}^{k} c_j (r-i\varphi)^{j} \right)
 \Re\left(P(r-i\varphi) \right)
 &\leq 
 P(r)\left(1-\frac{k(\varphi/r)^2}{12}\right)  \ \text{  for $r$ large and }|\varphi|\leq \pi.
 \label{eq:bound_P_and_ner}
\end{align}
We thus get
\begin{align}
|I^c_{n,2}| 
=&\,
2\left|
\frac{1}{2\pi e^{r}} \int_{\delta}^{\pi} \exp\bigl(v P(r-i\varphi) +ne^{-r}e^{i\varphi}+i\varphi+O(ve^{-r}e^{i\varphi})      \bigr) \, d\varphi \right|\nonumber\\
\ll&
e^{-r} \int_{\delta}^{\pi}  \exp\left(v\Re\left(P(r-i\varphi)\right)+ne^{-r}\cos(\varphi)  \right) \, d\varphi \nonumber
\\
\ll& 
\exp\left(vP(r)+ ne^{-r}\right)e^{-r}\int_{\delta}^{\pi}  \exp\left(-\frac{kvP(r)r^{-2}+ne^{-r}}{12}\varphi^2 \right) \, d\varphi.
%\ll& 
%\frac{\exp\left(vP(r+ ne^{-r}\right)}{2\pi\sqrt{b/6}} \int_{\delta\sqrt{b}}^{\infty}  \exp\left(-x^2/2 \right) \, d\varphi \\
%\ll&
%\frac{\exp\left(v\frac{r^{k+1}}{k+1}+ ne^{-r}\right)}{2\pi\sqrt{b/6}} e^{-\delta\sqrt{b}}
\label{eq:estimate_I2c}
\end{align}
We now have $kvP(r)r^{-2} = O(\log^{k-1}n) = o(ne^{-r})$ and thus 
\begin{align}
|I^c_{n,2}| 
&\ll
\exp\left(vP(r)+ ne^{-r}\right)e^{-r}\int_{\delta}^{\pi}  \exp\left(-\frac{ne^{-r}}{24}\varphi^2 \right) \, d\varphi\nonumber\\
&\ll
\frac{\exp\left(vP(r)+ ne^{-r}\right)}{e^{r}\sqrt{ne^{-r}}} 
\int_{\delta\sqrt{ne^{-r}}}^{\infty}  \exp\left(-\frac{x^2}{2} \right) \, dx\nonumber\\
&\ll
\frac{\exp\left(vP(r)+ ne^{-r}\right)}{e^{r}\delta\sqrt{ne^{-r}}}  e^{-\delta\sqrt{ne^{-r}}}.
\label{eq:estimate_I2c2}
\end{align}
Inserting the definition of $\delta$ and the asymptotic behaviour of $ne^{-r}$ shows that $I^c_{n,2}$ is of lower order. 
It remains to show that the integrals over $I_{n,1}$  and $I_{n,3}$ are also of lower order. The computations for both are almost the same and we thus only take a look at $I_{n,3}$.
We have 
\begin{align*}
|I_{n,3}| 
&\leq
 \frac{1}{2\pi}\left| \int_{e^{-r}}^\pi \exp\bigl(vg_\Theta(e^{-ix}) +nix\bigr) \, dx\right|
 \leq
 \frac{1}{2\pi} \int_{e^{-r}}^\pi \exp\bigl(\Re(vg_\Theta(e^{-ix})) \bigr) \, dx.
\end{align*}
We first consider the asymptotic behaviour of $g_\Theta(e^{-ix})$ as $x\to 0$. 
Equation \eqref{eq:g_theta_to_2} gives
\begin{align*}
g_\Theta(e^{-ix})
=&\,
P(-\log(x)-i\pi/2) +O(x).
\end{align*}
Using the Taylor expansion, we get for $x\to 0$
\begin{align*}
 \Re(g_\Theta(e^{-ix}))
 &=
 P(-\log(x)) - P''(-\log(x))\pi^2/8+ O\big(P^{(4)}(-\log(x))\big)  +O(x).
\end{align*}
Since $-\log(x)\geq 0$ for $x<1$, there exists a constant $0<c<1$ such that
\begin{align}
 \Re(g_\Theta(e^{-ix})
 &\leq 
 P(-\log(x)) - \frac{9}{8} P''(-\log(x))
 \ \text{ for all } \
 x\in]0,c].
 \label{eq:bound_gamma3}
\end{align}
We now spilt the integral into the integral over the regions 
$[e^{-r},c]$ and $[c,\pi]$. By assumption, $g_\Theta(t)$ is continuous in the punctured disc $\{|t|\leq 1$, $t\neq 1\}$.
We thus have clearly
\begin{align*}
 \frac{1}{2\pi} \int_{c}^\pi \exp\bigl(\Re(vg_\Theta(e^{-ix})) \bigr) \, dx = O(1).
\end{align*}
Furthermore, we get with the above estimates and the variable substitution $y=-\log(x)$
\begin{align*}
 \frac{1}{2\pi} \int_{e^{-r}}^c \exp\bigl(\Re(vg_\Theta(e^{-ix})) \bigr) \, dx
 &\leq 
 \frac{1}{2\pi} \int_{e^{-r}}^c \exp\left(vP(-\log(x)) - \frac{9}{8} vP''(-\log(x)) \right) \, dx\\
 &=
 \frac{1 }{2\pi} \int_{-\log(c)}^{r} \exp\Bigl(  vP(y) -  \frac{9}{8} vP''(y) \Bigr) e^{-y} \, dy.
\end{align*}
Thus we can apply Lemma~\ref{lem:int_exp_P} with $Q(y)=  vP(y) -  \frac{9}{8} vP''(y)- y$ and get 
\begin{align*}
|I_{n,3}| 
=
O\left(
\frac{\exp\Bigl(  vP(r) -  \frac{9}{8} vP''(r)- r  \Bigr)}{ P'(r)- \frac{9}{8} vP'''(r) -1}
\right).
\end{align*}
We now have to show that this is of lower order. Recall, the main term in the theorem is 
\begin{align}
 = \frac{\exp\left(vP(r)+ ne^{-r}-r\right) }{\sqrt{2\pi vP''(r)+2\pi ne^{-r}}}.
\end{align}
Using that
\begin{align}
r \sim \log(n/v), \ 
vP'(r) &=  ne^{-r} \sim v\log^{k}(n) \ \text{ and } \
   P''(r) \sim k\log^{k-1}(n) 
\end{align}
immediately completes the proof.
\end{proof}

We will see that we need in the Sections~\ref{sec:cycles}, \ref{sec:limit_shape} and~\ref{sec:dtv} also some slight generalizations of Theorem~\ref{thm:aux_asypmtotic}.
\begin{corollary}
\label{cor:aux_asypmtotic} 
 Let $g_\Theta(t)$ and $P(r)$ be as in Theorem~\ref{thm:aux_asypmtotic}. 
  Let further $f(t)$ be a holomorphic function with radius of convergence strictly bigger than $1$ and $f(1)\neq 0$. 
 We then have for $v \in [v_1 , v_2 ]$ with arbitrary constants $1\leq v_1<v_2<\infty$ %and $\Im(v) =o\left(r^\frac{-7k}{12}\right)$ 
 \begin{align*}
[t^n]\,f(t)\exp\bigl(vg_\Theta(t)\bigr)
= 
\frac{f(1)\exp\left(vP(r)+ ne^{-r}\right) }{\sqrt{2\pi vP''(r)+2\pi ne^{-r}}} \left(1+ O(\log^{-k/2}(n))+O\left(\Im(v)r^\frac{7k}{12}\right) \right),
\end{align*}
where $r$ is the solution of the equation
 \begin{align}
 vP'(r)= ne^{-r}.
 \label{eq:thm:aux_asypmtotic_saddle_eq2}
\end{align}
\end{corollary}

\begin{proof}
 The proof is almost the same as for Theorem~\ref{thm:aux_asypmtotic}. We thus describe only the necessary adjustments. 
In the integral $I_{n,2,\delta}$, one has to use the Taylor expansion of $f(t)$ around one. It is straight forward to see that only the term $f(1)$ gives a relevant contribution.
In the remaining integrals, we use the estimate $f(t)=O(1)$. This completes the proof.
\end{proof}

\begin{corollary}
\label{cor:aux_asypmtotic2} 
 Let $g_\Theta(t)$ and $P(r)$ be as in Theorem~\ref{cor:aux_asypmtotic}. 
 Let further $f(t)$ be a function such tat such that
 \begin{itemize}
   \item $f(t)$ is holomorphic for $|t|<1$,
   \item $f(t)$ is continuous in the punctured disc $\{|t|\leq 1$, $t\neq 1\}$ and 
   \item there is a $j\geq 0$ and a $c_f\in\C$ such that
   $$f(e^{-w}) = c_f \frac{\big(-\log(w)\big)^k}{w^j} + O\left( \frac{\big(-\log(w)\big)^{j-1}}{w^j} \right) \ \text{ as }w\to 0,\, \Re(w)\geq 0.$$
 \end{itemize}
We then have 
 \begin{align*}
[t^n]\,f(t)\exp\bigl(g_\Theta(t)\bigr)
= 
c_f\, r^k e^{jr} \,\frac{\exp\left(P(r)+ ne^{-r}\right) }{\sqrt{2\pi P''(r)+2\pi ne^{-r}}} \left(1+ O(\log^{-1/2}n)\right),
\end{align*}
where $r$ is the solution of the equation
 \begin{align}
  P'(r)= ne^{-r}.
 \label{eq:thm:aux_asypmtotic_saddle_eq3}
\end{align}
\end{corollary}

\begin{proof}
We use the same notation as in the proof of Theorem~\ref{thm:aux_asypmtotic} and describe only the necessary adjustments. 
We have 
\begin{align}
 I_n
 =
 \frac{1}{2\pi i} \int_\gamma f(t) \exp\bigl(vg_\Theta(t)\bigr) \frac{1}{t^{n+1}}\, dt
 = 
 \frac{1}{2\pi i} \int_{\gamma''}f(e^{-w}) \exp\bigl(vg_\Theta(e^{-w})\bigr) e^{nw}\, dw.
\end{align}
We use that $\gamma''_2=\{w=e^{-r}e^{i\varphi}, \varphi\in[-\pi,\pi]\}$ and obtain
\begin{align*}
I_{n,2,\delta}
=
\frac{1}{2\pi e^{r}} \int_{-\delta}^{\delta} f\left(e^{-e^{-r}e^{i\varphi}}\right) \exp\bigl(v P(r-i\varphi) +ne^{-r}e^{i\varphi}+i\varphi+O(ve^{-r}e^{i\varphi})      \bigr) \, d\varphi.
\end{align*}
As $f(t)$ has a singularity at $t=1$, one has to check if $f$ has a relevant influence to the saddle point equation.
However, it is not difficult to see that we can use the same $r$ as in Theorem~\ref{thm:aux_asypmtotic}.
Thus we immediately obtain that 
\begin{align*}
I_{n,2,\delta}
&=
 f\left(e^{-e^{-r}}\right) \,\frac{\exp\left(P(r)+ ne^{-r}\right) }{\sqrt{2\pi P''(r)+2\pi ne^{-r}}} \left(1+ O(\log^{-1/2}n)\right)\\
&=
c_f\, r^k e^{jr}\,\frac{\exp\left(P(r)+ ne^{-r}\right) }{\sqrt{2\pi P''(r)+2\pi ne^{-r}}} \left(1+ O(\log^{-1/2}n)\right).
\end{align*}
The remaining parts of $I_2$ are of lower order. This completes the proof.
\end{proof}

\section{Asymptotic statistics of cycles}
\label{sec:cycles}

We apply in this section Theorem~\ref{thm:aux_asypmtotic} to determine the asymptotic behaviour of various random variables on $\Sn$.

\subsection{Assumptions on the cycle weights $\theta_m$}
\label{sec:assumption_theta}

Theorem~\ref{thm:aux_asypmtotic} requires only the analytic properties of $g_\Theta(t)$, 
but does not require that $\theta_m = \log^k(m)$. 
In particular, it follows immediately with Lemma~\ref{lem:asymptotic_g} that the generating function $\widetilde{g}_\Theta(t)$ corresponding to cycle weights
\begin{align}
 \tilde{\theta}_m = \log^k(m) +\sum_{j=0}^{k-1} a_j \log^j(m) \ \text{ with } a_j\in\R \text{ for all }j
 \label{eq:alternative_cycle_weights}
\end{align}
has the same analytic properties, but with a slightly different polynomial $\tilde{P}$.
We thus can apply Theorem~\ref{thm:aux_asypmtotic} also for the cycle weights in \eqref{eq:alternative_cycle_weights}.
We thus assume from now only that we have $\theta_m\geq 0$ for all $m\geq 1$ and that the corresponding generating series $g_\Theta(t)$ in \eqref{eq:def_g_theta} is:
\begin{itemize}
 \item holomorphic for $|t|<1$,
 \item continuous in the punctured disc $\{|t|\leq 1$, $t\neq 1\}$ and 
 \item that there exists a polynomial $P$ with 
\begin{align}
 P(r)= \frac{r^{k+1}}{k+1}+ \sum_{j=0}^{k} c_j r^j
\end{align}
with $k\geq 1$
such that  
 \begin{align}
  g_\Theta(e^{-w})
    =
  P\bigl(-\log(w)\bigr)
  +
  O(w) 
  \ \text{ as }w\to 0,\, \Re(w)\geq 0.
  \label{eq:g_theta_to_3}
 \end{align}
\end{itemize}
Further, we define $r=r_{n,\Theta,v}$ to be a solution of the saddle point equation \eqref{eq:thm:aux_asypmtotic_saddle_eq}, i.e.
\begin{align}
 vP'(r)= ne^{-r}.
 \label{eq:thm:aux_asypmtotic_saddle_eq4}
\end{align}

\subsection{Normalisation Constant $h_n$}\label{sec:hn}

Recall, we have seen in Corollary~\ref{cor:Hn_generatingN} that
 \begin{align}
  \sum_{n=0}^\infty h_n t^n
  =
  \exp\bigl(g_\Theta(t)\bigr),
  \label{eq:Hn_generatingN2}
 \end{align}
where $h_n$ is the normalisation constant of the measure $\mathbb{P}_\Theta$ in Definition~\ref{def:Pb_measure}.
We thus immediately get with Theorem~\ref{thm:aux_asypmtotic} that
 \begin{align}
h_n=
\frac{\exp\left(P(r)+ ne^{-r}\right) }{e^{r}\sqrt{2\pi P''(r)+2\pi ne^{-r}}} \left(1+ O(\log^{-k/2}(n)) \right),
 \end{align}
where $P$ is as in Section~\ref{sec:assumption_theta} and $r$ is the solution of the equation $P'(r)= ne^{-r}$.

\subsection{Cycle counts}\label{sec:cyclecounts}

Our first result deals with the asymptotics of the cycle counts $C_m$'s (i.e., the numbers of cycles of length $m\in\N$, respectively, in a
random permutation $\sigma\in\Sn$).

\begin{theorem}
\label{thm:limit_theorem_cycles}
Suppose that $\Theta = (\theta_m)_{m\in\N}$ fulfils the assumptions in Section~\ref{sec:assumption_theta} and that $\Sn$ is endowed with $\mathbb{P}_\Theta$.
We then have for each $b\in\N$ as $n\to\infty$
\begin{align}
  \left(C_1, C_2,  \cdots,C_b \right)
  \stackrel{d}{\to}
  \left(Y_1,\cdots,Y_b\right)
\end{align}
with $Y_1,\cdots,Y_b$ independent Poisson distributed random variables with $\E{Y_m} = \frac{\theta_m}{m}$.
\end{theorem}
\begin{proof}
Using Lemma~\ref{lem:cycle_index_theorem} it is forthright to see that we have
\begin{align}
  \sum_{n=0}^\infty
   h_n \ET{\exp \left(i\sum_{m=1}^b s_m C_m \right)} t^n
    &=
\exp\left( \sum_{m=1}^b \frac{\theta_m}{m} (e^{is_m}-1) t^m \right) \exp\big(g_\Theta(t)\big)
\label{eq:generating_cycles}
\end{align}
as formal power series in $t$. The details of this computation can be found for instance in \cite[Theorem~3.1]{NiZe11}.
Corollary~\ref{cor:aux_asypmtotic} with $v=1$ then immediately implies that
\begin{align}
 \ET{\exp \left(i\sum_{m=1}^b s_m C_m \right)}
 =
 \exp\left( \sum_{m=1}^b \frac{\theta_m}{m} (e^{is_m}-1) \right)\left(1+ O(\log^{-k/2}(n)\right).
\end{align}
The theorem now follows immediately from L\'evy's continuity theorem.% \textcolor{red}{REFERENCE}.
\end{proof}
\subsection{Total number of cycles}
\label{sec:tot_num_cycles}

We denote by $K_{0n}$ the total number of cycles in the cycles decomposition of $\sigma\in\Sn$, i.e.
\begin{align}
 K_{0n}:= \sum_{m=1}^n C_m.
\end{align}

\begin{theorem}
\label{thm:limit_theorem_total_cycles}
Suppose that $\Theta = (\theta_m)_{m\in\N}$ fulfils the assumptions in Section~\ref{sec:assumption_theta} and that $\Sn$ is endowed with $\mathbb{P}_\Theta$.
We then have
\begin{align}
  \frac{K_{0n} - \E{K_{0n}}}{\sqrt{\frac{\log^{k+1}(n)}{k+1}}}
  \stackrel{d}{\to}
  \mathcal{N}(0,1)
\end{align}
where $\mathcal{N}(0,1)$ is the standard normal distribution and $\E{K_{0n}} \sim \frac{\log^{k+1}(n)}{k+1}$.
\end{theorem}

\begin{proof}
We have for each  $s\in\C$ as formal power series in $t$
\begin{align}
\ET{\exp\bigl( s K_{0n}  \bigr)}
=
\ET{\exp\left( s \sum_{m=1}^n C_m  \right)} 
=
\frac{1}{h_n}[t^n]
\exp\bigl( e^s g_\Theta (t) \bigr).
\label{eq:gen_K0n}
\end{align}
This equation follows immediately from Lemma~\ref{lem:cycle_index_theorem}.
The exact details of this computation can be found for instance in \cite[Lemma~4.1]{NiZe11}.
Although  the  expressions  in  \eqref{eq:gen_K0n}  holds  for general $s\in \C$, 
we will calculate the asymptotic behaviour of the moment  generating function of $K_{0n}$ only on the positive half-line $s\geq 0$. 
Theorem 2.2 in \cite{Ca07} shows that this is enough to prove statement of the theorem.
Using Theorem~\ref{thm:aux_asypmtotic}, we have 
\begin{align}
[t^n]\exp\bigl( e^s g_\Theta (t) \bigr)
&= 
\frac{\exp\left(e^sP(r)+ ne^{-r}\right) }{e^{r}\sqrt{2\pi e^s P''(r)+2\pi ne^{-r}}} \left(1+ O(\log^{-k/2}(n)) \right)\nonumber\\
&= 
\frac{\exp\left(e^sP(r)+  e^sP'(r)\right) }{e^{r}\sqrt{2\pi e^s P''(r)+2\pi  e^sP'(r)}} \left(1+ O(\log^{-k/2}(n)) \right)
 \label{eq:thm:aux_asypmtoticK0n}
\end{align}
where $r$ is a solution of the equation
\begin{align}
 e^sP'(r)= ne^{-r}.
\end{align}
Equation \eqref{eq:saddle_solution_asympt} now implies that for $s$ bounded we have 
\begin{align}
 r \sim \log(n) -s, \ P(r) \sim  \frac{\log^{k+1}(n)}{k+1},\ P'(r) \sim  \log^{k}(n)    \  \text{ and } \ \ P''(r) =k \log^{k-1}(n).
\end{align}
As \eqref{eq:thm:aux_asypmtoticK0n} hold uniformly of $s$ bounded, we can replace $s$ by $\tilde{s} = \frac{s}{\sqrt{\frac{\log^{k+1}(n)}{k+1}}}$.
A direct computation then shows that 
\begin{align}
\frac{1}{h_n}[t^n]\exp\bigl( e^{\tilde{s}} g_\Theta (t) \bigr)
&= 
e^{s^2/2} \left(1+ O(\log^{-k/2}(n)) \right).
\end{align}
This completes the proof of the theorem.
\end{proof}
\begin{remark}
 We can determine with Theorem~\ref{thm:aux_asypmtotic} the asymptotic behaviour of $\ET{\exp\bigl( s K_{0n}  \bigr)}$ for $s\geq 0$.
 If we could extend Theorem~\ref{thm:aux_asypmtotic} and compute $\ET{\exp\bigl( s K_{0n}  \bigr)}$ also for $s\in\C$ then this would  
 imply immediately much stronger results, see for instance \cite{NiZe11}. 
\end{remark}

\subsection{Lexicographic ordering of cycles}
\label{sec:unordered_cycles_sub}

Often cycles  in the cycle decomposition of a permutation are ordered by length.
Another convenient way is to list the cycles (and their lengths) via the
lexicographic ordering, that is, by tagging them with a suitable
increasing subsequence of elements starting from $1$.
\begin{definition}
\label{def:L-lex}
For permutation $\sigma\in\Sn$ decomposed as a product of cycles,
let $L_1 = L_1(\sigma)$ be the length of the cycle containing
element~$1$, $L_2=L_2(\sigma)$ the length of the cycle containing
the smallest element not in the previous cycle, etc. The sequence
$(L_j)$ is said to be \emph{lexicographically ordered}.
\end{definition}
Our next aim is to determine the asymptotic behaviour of the \emph{lexicographically ordered} cycles lengths.
For this we have to extend the assumptions in Section~\ref{sec:assumption_theta} a little bit.
We assume in addition that we have for all $j\geq 1$
 \begin{align}
  g_\Theta^{(j)}(e^{-w})
    =
  (j-1)!\left(\frac{e^{w}}{w}\right)^j \big(-\log(w)\big)^k
  +
  O\left( \left(\frac{e^{w}}{w}\right)^j \big(-\log(w)\big)^{k-1} \right),
  \label{eq:g_theta_to_lexi}
 \end{align}
where $g_\Theta^{(j)}(t)$ is the $j$'th derivative of $g_\Theta$. 
If the function $g_\Theta(t)$ fulfils the assumptions in Section~\ref{sec:assumption_theta} and can be analytically extended beyond the punctured disc $\{|t|\leq 1$, $t\neq 1\}$ then 
the assumption \eqref{eq:g_theta_to_lexi} is automatically fulfilled. 
For concreteness, let us define the following region,
\begin{definition}
\label{def_delta_0}
Let $1 < R$ and $0 < \phi <\frac{\pi}{2}$ be given. We then define
\begin{align}
\Delta_0 = \Delta_0(R,\phi) = \{ t\in \C ; |t|<R, z \neq 1 ,|\arg(z-1)|>\phi\}.
\label{eq_def_delta_0}
\end{align}
\end{definition}
An illustration of $\Delta_0(R,\phi)$ can be found in Figure~\ref{fig:delta_0}
\begin{figure}[ht!]
 \centering
 \includegraphics[width=.3\textwidth]{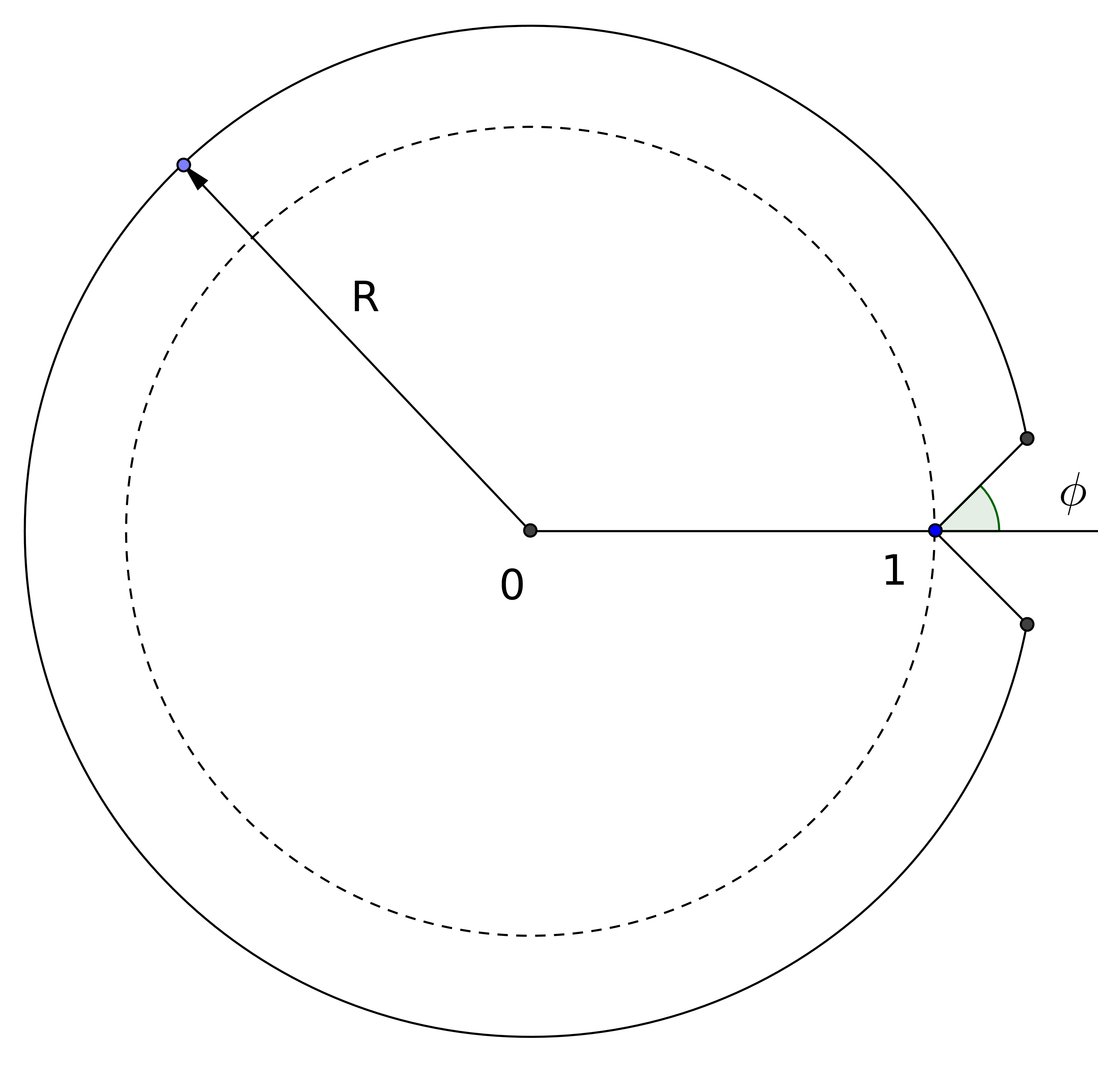}
 \caption{Illustration of $\Delta_0$}
 \label{fig:delta_0}
\end{figure}
We then have 
\begin{lemma}
\label{lem:derivative_of_g}
 Suppose that $\Theta = (\theta_m)_{m\in\N}$ fulfills the assumptions in Section~\ref{sec:assumption_theta} and that $g_\Theta(t)$ can be analytically extended 
 to some $\Delta_0(R,\phi)$. Then the assumption \eqref{eq:g_theta_to_lexi} is  fulfilled.
\end{lemma}
The lemma follows immediately with Cauchy's integral formula for higher order derivatives.
Lemma~\ref{lem:asymptotic_g} and~\ref{lem:derivative_of_g} immediately imply that
$$
 g_\Theta(t)=\sum_{m=1}^\infty \frac{\log^k m}{m} t^m
 $$
fulfils the assumption~\eqref{eq:g_theta_to_lexi}. 
We now can show
\begin{theorem}
\label{thm:L1L2}
Suppose that $\Theta = (\theta_m)_{m\in\N}$ fulfils the assumptions in Section~\ref{sec:assumption_theta} and the assumption \eqref{eq:g_theta_to_lexi}. 
If $\Sn$ is endowed with $\mathbb{P}_\Theta$, we then have for each $b\in\N$
\begin{align}
  \left(\frac{L_1\cdot r^k}{n}, \frac{L_2\cdot r^k}{n},\ldots, \frac{L_b\cdot r^k}{n} \right) 
 \stackrel{d}{\longrightarrow}
 (E_1,E_2,\ldots, E_b),
\end{align}
where $(E_m)_{m=1}^b$ are iid exponential distributed random variables with parameter $1$.
\end{theorem}

\begin{proof}
 We prove first the case $b=1$. We have
 \begin{align}
  \Pb{L_1= m} = \frac{\theta_m}{n} \frac{h_{n-m}}{h_n}.
  \label{eq:l1_dist}
 \end{align}
The proof of \eqref{eq:l1_dist} can be found for instance in \cite[Proposition~2.1]{ErUe11}.
We now claim that we have for each $j\in\N$
\begin{align}
 \ET{(L_1-1)_j} 
%  &=
%  \ET{(L_1-1)(L_1-2)\cdots(L_1-j)} \nonumber\\
 &=
 \frac{1}{nh_n}[t^n]\, t^j g_\Theta^{(j+1)}(t) \exp\left( g_\Theta(t) \right),
\end{align}
where $(x)_j = x(x-1)\cdots (x-j+1)$  denotes the falling factorial.
Indeed, using \eqref{eq:l1_dist} gives
\begin{align*}
 \ET{(L_1-1)_j} 
 &=
 \frac{1}{h_n}
 \sum_{m=1}^n
 (m-1)_j \cdot\Pb{L_1= m} 
 =
  \frac{1}{n h_n}
 \sum_{m=1}^n
 (m-1)_j \cdot\theta_m h_{n-m}\\
 &=
 \frac{1}{n h_n}
 [t^n] \left( \sum_{m=1}^n (m-1)_j\cdot\theta_m t^m\right) \exp(g_\Theta(t))
 =
  \frac{1}{nh_n}[t^n]\, t^j g_\Theta^{(j+1)}(t) \exp\left( g_\Theta(t) \right).
 \end{align*}
We now can use Corollary~\ref{cor:aux_asypmtotic2} together with with assumption \eqref{eq:g_theta_to_lexi} to compute $ \ET{(L_1-1)_j} $. 
We obtain
\begin{align*}
   \ET{(L_1-1)_j} 
   &=
  \frac{j!}{n} r^k e^{(j+1)r} (1+o(1))
    =
  j! \frac{n^j r^k}{(ne^{-r})^{j+1}} (1+o(1))\\
  &= 
  j! \left(\frac{n}{r^k}\right)^j (1+o(1)),
\end{align*}
where we have used on the last line that $ne^{-r} = P'(r) \sim r^k$.
This immediately implies with a simple induction that
\begin{align}
   \ET{\left(\frac{L_1\cdot r^k}{n}\right)^j} 
  &= 
  j! \left(\frac{n}{r^k}\right)^j(1+o(1)).
   \label{eq:forL1:last}
\end{align}
Since \eqref{eq:forL1:last} holds for each $j\in\N$, we get that $\frac{L_1\cdot r^k}{n}$ converges in distribution to exponential distributed random variables with parameter $1$.
This completes the case $b=1$. 
 The proof of the case $b>1$ is similarly and we thus omit it. However, the interested reader can find more details for instance in \cite[Lemma~5.7]{BoZe14} or in \cite{StZe14a}.
\end{proof}

\section{Limit shape}
\label{sec:limit_shape}

We consider in this section the shape of Young diagrams associated to random permutations and 
study the typical behavior as $n\to\infty$ with respect to the measure $\mathbb{P}_\Theta$ under the assumptions in Section~\ref{sec:assumption_theta}.
We show that this shape converges to a limit shape and that fluctuations near a point of this limit shape behave like a normal random variable.
In this section we shall mainly follow the techniques from \cite{CiZe13}. We first define
\begin{align}
\label{eq:def_shape_fkt}
w_n(x)
=
\sum_{k\geq x}C_k.
\end{align}
The function $w_n(x)=w_n(x,\sigma)$ is as a function in $x$  piecewise constant and right continuous. 
Further $w_n(x,\sigma)$ can be interpreted as the upper boundary of the Young diagram corresponding the cycle type of the permutation $\sigma$.
A detailed illustration of this can be found in \cite[Section~1]{CiZe13}.

The limit shape of the process $w_n(x)$ as $n\to\infty$ with of the respect to probability measures $\mathbb{P}_\Theta$ on $\Sn$ 
(and sequences of positive real numbers $\overline{n}$ and $n^*$ with $\overline{n} \cdot n^* =n$) is understood as a function $w_\infty:\R^+ \to \R^+$ 
such that for each $\eps,\delta>0$
\begin{align}
\label{eq:def_limit_shape}
\lim_{n\to+\infty} \PT{\{\sigma\in\Sn:\, \sup_{x\geq \delta} |(\overline{n})^{-1} w_n (xn^*)- w_\infty(x)| \leq \epsilon \}  }
= 1.
\end{align}
The assumption $\overline{n} \cdot n^* =n$ ensures that the area under the rescaled Young diagram is $1$.
One of the most frequent choices is $\overline{n}=n^*=n^{1/2}$, however this is often not the optimal choice.
The computations in Section~\ref{sec:cycles} suggests that  the length of a typical cycle has order of magnitude $n/r^k$.
It is thus natural to choose
\begin{align}
 n^*
 = 
 \frac{n}{r^k}
 \ \text{ and } \
 \overline{n}
 = 
 r^k
 \label{eq:def_n*}
\end{align}
with $r$ the solution of the equation \eqref{eq:thm:aux_asypmtotic_saddle_eq}.

The next natural question is then whether fluctuations satisfy a central limit theorem, namely whether 
$$
(\overline{n})^{-1}  w_n(n^* x)-w_\infty(x)
$$
converges for a given $x$ (after centering and applying normalization) to a normal distribution. Also it is natural to ask
if the process converges in distribution to a Gaussian process on the space of c\`adl\`ag functions. 
Of course the role of the probability measure on $\Sn$ is important for that.

We first consider the behavior for a given $x>0$. We have
\begin{theorem}
\label{thm:limit_shape_saddle}
Let $k\geq 3$ and $n^*$ and $\overline{n}$ be as in \eqref{eq:def_n*} and suppose that $\mathbb{P}_\Theta$ fulfills the assumptions in Section~\ref{sec:assumption_theta}.
We then have the following results.
\begin{enumerate}
 \item The limit shape exists for the process $w_n(x)$ as $n\to\infty$ with the scaling $n^*$ and $\overline{n}$ as in \eqref{eq:def_n*}
 and the limit shape is given by
 \begin{align}
   w_\infty(x) =  \int_{x}^\infty  u^{-1} e^{-u}\,du.
   \label{eq:def_w_infinity}
 \end{align} 
 \item The fluctuations at a point $x$ of the limit shape behave like
 \begin{eqnarray}&&
\widetilde w_n(x)
:=
\frac{w_n(xn^*)-\overline n\left( w_\infty(x)+z_n(x)\right)}{(\overline n)^{1/2}}
\stackrel{\mathcal L}{\longrightarrow}
\mathcal N\left(0, \sigma_\infty^2(x)\right)
\label{eq:CLT_Pt}
\end{eqnarray}
with
 $$
 \sigma_\infty^2(x):=e^{-2x}+w_\infty(x)
 $$
 and $z_n=O(1/\log n)$.
 \end{enumerate}
\end{theorem}
\begin{remark}
The condition $k\geq 3$ is required in the estimates used for the error terms.
However, we believe that this condition could be relaxed to $k\geq1$ by a more detailed investigation of the corresponding error terms.
\end{remark}

We prove this theorem by computing the Laplace transform of $w_n(x)$.
We have
\begin{lemma}
\label{lem:relies2}
Let $k\geq 3$. We have for bounded $s\geq 0$ and with respect to $\mathbb{P}_\Theta$ as $n\to\infty$
$$
\ET{\mathrm{exp}\bigl(-s\widetilde w_n(x)\bigr)}={\sigma_\infty^2(x)}\frac{s^2}{2}+O\big((\overline n)^{-\frac{1}{2}}s^3\big). 
$$
\end{lemma}
We will give the proof of Lemma~\ref{lem:relies2} in Section~\ref{sec:proof_of_lemma52}.
However, we show immediately that Lemma~\ref{lem:relies2}  implies Theorem~\ref{thm:limit_shape_saddle}. 
The structure of the proof is similar to the one appearing in \cite{CiZe13}, and we give the proof for the convince of the reader.

\begin{proof}[Proof of Theorem~\ref{thm:limit_shape_saddle}]
Theorem~2.2 in \cite{Ca07} shows that it is sufficient to compute the Laplace transform for $s\geq0$ to establish the CLT. 
Therefore Lemma~\ref{lem:relies2} immediately implies the second point of Theorem~\ref{thm:limit_shape_saddle}.
Thus it remains to show that that $w_\infty(x)$ is the limit shape.
Let $\eps>0$ be arbitrary and choose $0=x_0 <x_1<\dots<x_\ell$ such that 
$$w_\infty(x_{j+1})-w_\infty(x_{j})<\eps/2 \text{ for }1\leq j \leq \ell-1 \text{ and } w_\infty(x_{\ell})<\eps/2.$$ 
% for $1\leq j \leq \ell-1$ and $w_\infty(x_{\ell})<\eps/2$.
We now claim that we have for each $x\in\R^+$
\begin{align}
|(\overline n)^{-1}w_n(xn^*)-w_\infty(x)|>\eps 
\ \Longrightarrow \ 
\exists j \text{ with }|(\overline n)^{-1}w_n(x_jn^*)-w_\infty(x_j)|>\eps/2.
\label{eq:Limit_shape_exists_proof}
\end{align}
Indeed, let us for consider first the case $(\overline n)^{-1}w_n(x^*)-w_\infty(x)>\eps$.
Clearly, there exists a $j$ such that $x_j\leq x\leq x_{j+1}$. 
Since $w_n(x)$ is a monotone decreasing function, we get immediately
\begin{align*}
(\overline n)^{-1}w_n(xn^*)-w_\infty(x)>\eps  
\ &\Longrightarrow \ 
 (\overline n)^{-1}w_n(x_jn^*)-w_\infty(x)>\eps\\
 \ &\Longrightarrow \ 
  (\overline n)^{-1}w_n(x_jn^*)-w_\infty(x_j)>\eps+ w_\infty(x) -w_\infty(x_j)\\ 
  &\Longrightarrow \ 
  (\overline n)^{-1}w_n(x_jn^*)-w_\infty(x_j)>\eps/2.
\end{align*}
The computation in the second case is similar.
Using \eqref{eq:Limit_shape_exists_proof}, we obtain
\begin{align}
\label{eq:proof_limit_shape_Pt}
  \PT{\sup_{x\geq 0} |\overline n w_n (x^*)- w_\infty(x)| \geq \eps}  
&\leq 
\sum_{j=1}^\ell \PT{|\overline n w_n (x_j^*)- w_\infty(x_j)| \geq \eps/2 }  .
\end{align}
It now follows from \eqref{eq:CLT_Pt} that each summand in \eqref{eq:proof_limit_shape_Pt} tends to $0$ as $n\to \infty$. 
This completes the proof.
\end{proof}

We are also interested in the joint behaviour at different points of the limit shape.
For this, let $x_\ell\geq x_{\ell-1}\geq \dots\geq x_1\geq 0$ be given.
From computational point of view, it is easier to study the increments. 
We thus consider
\begin{equation}\label{eq:increm}
\mathbf{w}_n(\textbf{x})=\bigl(w_n(x_\ell),\,w_n(x_{\ell-1})-w_n(x_\ell),\,\ldots,\,w_n(x_1)-w_n(x_2)\bigr).
\end{equation}
We now have 
\begin{theorem}
\label{thm:beh_increments_saddle}
For $\ell\geq 2$ and $x_\ell\geq x_{\ell-1}\geq \dots\geq x_1\geq 0$, let 
\begin{align}
 \widetilde{\mathbf{w}}_n(\mathbf x)=\bigl(\widetilde w_n(x_\ell),\,\widetilde w_n(x_{\ell-1})-\widetilde w_n(x_\ell),\,\ldots,\,\widetilde w_n(x_1)-\widetilde w_n(x_2)\bigr)
\end{align}
with $\widetilde w_n$ as in Theorem~\ref{thm:limit_shape_saddle}.
Set $x_{\ell+1}=+\infty$. We then have for $1\leq j<i<\ell$ 
\begin{align}
\widetilde w_\infty(x_i,\,x_j)
&:=
\lim_{n\to+\infty}\mathrm{Cov}\left(\widetilde w_n(x_j)-\widetilde w_n(x_{j+1}),\,\widetilde w_n(x_i)-\widetilde w_n(x_{i+1})\right)\\
&=
% \frac{\left(\Gamma(\alpha+1,\,x_i)-\Gamma(\alpha+1,\,x_{i+1})\right)\left(\Gamma(\alpha+1,\,x_{j})-\Gamma(\alpha+1,\,x_{j+1})\right)}{\Gamma(\alpha+1)\Gamma(\alpha+2)}.
(e^{-x_j}- e^{-x_{j+1}})(e^{-x_i}- e^{-x_{i+1}})(1+O(1/r)).
\nonumber
\end{align}
\end{theorem}
The proof of this theorem is given in Section~\ref{sec:proof_increments}.
Furthermore, we can extend Theorem~\ref{thm:beh_increments_saddle} to a functional CLT.
\begin{theorem}
\label{thm:func_CLT_saddle}
The process $\widetilde{w}_n:\R^+\to\R$ 
\textup{(}see Theorem~\ref{thm:limit_shape_saddle}\textup{)} 
converges weakly with respect to $\mathbb{P}_\Theta$ as $n\to\infty$ to a 
continuous Gaussian process $\widetilde{w}_\infty:\R^+\to\R$. Explicitly, we have $\widetilde{w}_\infty(x)\sim \mathcal N\big(0,\,\left(\sigma_\infty(x)\right)^2\big)$ and 
covariance structure is given in Theorem~\ref{thm:beh_increments_saddle}.
In particular, the increments are \emph{not} independent. 
\end{theorem}
The proof of this theorem is given in Section~\ref{sec:proof_func_CLT}.

\subsection{Proof of Lemma~\ref{lem:relies2}}
\label{sec:proof_of_lemma52}

We begin with some preparations. We have
\begin{lemma}[{\cite[Lemma~4.1]{CiZe13}}]
\label{lem:generating_w_n}
We have for $x \geq 0$ and $s \in\C$ 
\begin{align}
\label{eq:generating_w_n}
\ET{\mathrm{exp}\bigl(-s w_n(x)\bigr)}
&= 
\frac{1}{h_n}
[t^n] \left[\mathrm{exp}\left( g_\Theta(t) + (e^{-s}-1)\sum_{m= \floor{x}} ^\infty \frac{\theta_m}{m}t^m \right) \right]\\
&=
\frac{1}{h_n}
[t^n] \left[\mathrm{exp}\left( g_\Theta(t) + (e^{-s}-1)\sum_{m= \floor{x}} ^n \frac{\theta_m}{m}t^m \right) \right].
\end{align}
\end{lemma}

\begin{remark}\label{rem:s_positive}
Although the expressions in Lemmas~\ref{lem:generating_w_n}
holds in broader generality, we will calculate moment generating function of $w_n(x)$  on the positive half-line $ s\geq 0$ only. 
Theorem 2.2 in \cite{Ca07} shows that this is sufficient to establish Theorem~\ref{thm:limit_shape_saddle}.
\end{remark}
Furthermore, we need
\begin{lemma}
\label{lem:eulermac1}
Let $r$ be as in \eqref{eq:thm:aux_asypmtotic_saddle_eq},
 $n^*$ and $\overline{n}$ be as in \eqref{eq:def_n*}, $v=O(r^{-k/2})$, $q>0$, $j\in\Z$ and $x>0$.
We then define $r':=r+v$ and get 
 \begin{align}
  \sum_{m\geq xn^*} \theta_m m^j \exp(-mqe^{-r'})
  =
   r^k (e^r)^{j+1}\int_{x}^\infty  u^j \exp(-qu)\,du+ O( r^{k-1/2} (e^r)^{j+1}).
   \label{eq:lem:eulermac1}
 \end{align}
\end{lemma}
\begin{proof}
One can now show that the assumptions in Section~\ref{sec:assumption_theta} implies 
\begin{align}
 \theta_m = \log^k(m) +\sum_{j=0}^{k-1} a_j \log^j(m) + o(1) \ \text{ for some } a_j\in\R.
 \label{eq:alternative_cycle_weights2}
\end{align}
To see this, one can use classical singularity analysis, see \cite[Section~VI. 4]{FlSe09} or proceed backwards in the proof of Lemma~\ref{lem:asymptotic_g} using the properties of the Mellin transform.
Thus is is sufficient to study the case $\theta_m =\log^k (m)$ and to show that
 \begin{align}
  \sum_{m\geq xn^*} \log^k (m) m^j \exp(-mqe^{-r'})
  =
   r^k (e^r)^{j+1}\int_{x}^\infty  u^j e^{-qu}\,du+ O( r^{k-1/2} (e^r)^{j+1}).
   \label{eq:helpsum_mac}
 \end{align}
We apply Euler's summation formula to the sum on the LHS in \eqref{eq:lem:eulermac1} with $\theta_m =\log^k (m)$ and $f(y)=\log^k(y) y^j \exp(-yqe^{-r'})$.
This gives
\begin{align}
  \sum_{m\geq xn^*} \log^k (m) m^j \exp(-mqe^{-r'})
  =\,&  
  \int_{xn^*}^\infty \log^k(y) y^j \exp(-yqe^{-r'})\,dy   \nonumber\\
  &+
  \int_{xn^*}^\infty \left(y- \lfloor y\rfloor \right) f'(y)\,dy - f(xn^*)(xn^*-[xn^*])
\label{eq:Eulersum1}
\end{align}
with $\lfloor y\rfloor= \max\{m\in \N;\, m\leq y\}$. We first look at the integral
 \begin{align}
  \int_{xn^*}^\infty \log^k(y) y^j \exp(-yqe^{-r'})\,dy.
 \end{align}
We now use the variable substitution $u= ye^{-r'}$ and get
\begin{align*}
 \int_{xn^*}^\infty \log^k(y) y^j \exp(-yqe^{-r'})\,dy
 &=
 (e^{r'})^{j+1}\int_{e^{-r'}xn^*}  \log^k\big(ue^{r'}\big) u^j \exp(-qu)\,du\\
 &=
 (e^{r'})^{j+1}\int_{e^{-r'}xn^*}  \big(\log(u)+r'\big)^k u^j \exp(-qu)\,du.
\end{align*}
Using that $n^*=n/r^k$ and that $P'(r) =ne^{-r}$, we immediately obtain that 
\begin{align*}
 \int_{xn^*}^\infty \log^k(y) y^j \exp(-ye^{-r})\,dy
 &=
 r^k (e^r)^{j+1}\int_{x}  u^j \exp(-u)\,du+ O( r^{k-1/2} (e^r)^{j+1}).
\end{align*}
This gives the desired asymptotic behaviour.
 We thus have to show that the remaining terms in \eqref{eq:Eulersum1} are of lower order.
 We have
 \begin{align*}
  f'(y)
  =
  \big(1+ j\log(y)  -yqe^{-r'}  \log(y)\big)
  \log^{k-1}(y) y^{j-1} \exp(-yqe^{-r'}).
 \end{align*}
Thus we can use the same computation as for the main term for the integral over $f'(y)$ in \eqref{eq:Eulersum1} and immediately get that it is of lower order.
Further, inserting the definition of $n^*$ into $f(xn^*)(xn^*-[xn^*])$ also shows that it is of lower order. 
\end{proof}

\begin{proof}[Proof of Lemma~\ref{lem:relies2}]
Using the definition of $\widetilde w_n(x)$ in \eqref{eq:CLT_Pt}, we obtain  
\begin{align}
\ET{\mathrm{exp}\bigl(-s\widetilde w_n(x)\bigr)}
&= 
\mathrm{exp}\bigl(s (\overline n)^{1/2} (w_\infty(x)+z_n(x))   \bigr)\ET{\mathrm{exp}\bigl(-s^* w_n(xn^*)\bigr)}
\label{eq:laplace_w_n_limit_shape}
\end{align}
with $w_n(x)$ as in \eqref{eq:def_shape_fkt}, $n^*$ and $\overline n$ as in \eqref{eq:def_n*} and $s^*:=s(\overline n)^{-1/2}$.
Thus it is enough to compute the asymptotic behaviour of  $\ET{\mathrm{exp}\bigl(-s^* w_n(xn^*)\bigr)}$.
To do this, we apply Cauchy's integral formula to \eqref{eq:generating_w_n} and replace $x$ by $xn^*$ and $s$ by $s^*$ in \eqref{eq:generating_w_n}. 
This gives
\begin{align*}
h_n\ET{\mathrm{exp}\bigl(-s^* w_n(xn^*)\bigr)}
&= 
[t^n] \left[\mathrm{exp}\left( g_\Theta(t) + (e^{-s^*}-1)\sum_{m= \floor{xn^*}}^n \frac{\theta_m}{m}t^m \right) \right]\\
&=
\frac{1}{2\pi i} \oint_{\gamma} \mathrm{exp}\left( g_\Theta(t) + (e^{-s^*}-1)\sum_{m= \floor{xn^*}}^n \frac{\theta_m}{m}t^m \right)\, \frac{1}{t^{n+1}} dt,
% \label{eq:proof_of_limitshape1}
\end{align*}
where $\gamma$ is the contour $\gamma:=\{t=e^{-1/2}\cdot e^{i\varphi},\varphi\in[-\pi,\pi]\}$.
We now use a similar argumentation as in the proof of Theorem~\ref{thm:aux_asypmtotic}.
Applying the variable substitution $t=e^{-w}$, we get 
\begin{align}
 I_{n}
 = 
 \frac{1}{2\pi i} \int_{\gamma'} \exp\left(g_\Theta(e^{-w})  + (e^{-s^*}-1)\sum_{m= \floor{xn^*}}^n \frac{\theta_m}{m}e^{-mw}   \right) e^{nw}\, dw
 \label{eq:proof_of_limitshape2}
\end{align}
 with $\gamma':=\{t=1/2 +is,\,s\in[-\pi,\pi]\}$.
Note that the integrand in \eqref{eq:proof_of_limitshape2} is $2\pi i$ periodic. We thus can shift the contour $\gamma'$ to the contour $\gamma''$ with 
 \begin{align*}
\gamma''_1&:=\{w=(-\pi+u)i,u\in[0,\pi-e^{-r'}]\},\\
\gamma''_2&:=\{w=e^{-r'}e^{i\varphi},\varphi\in[-\pi/2 ,\pi/2]\},\\
\gamma''_3&:=\{w=iu ,u\in[e^{-r'},\pi]\},
 \end{align*}
where $r'>0$ will be determined below. 
We thus can write $I_n = I_{n,1}+ I_{n,2}+I_{n,3}$, where $ I_{n,j}$ corresponds to the integral over $\gamma''_j$.
We will show that $I_{n,2}$ is the leading term and that $I_{n,1}$ and $I_{n,3}$ are of lower order.

We begin by computing $I_{n,2}$. We have
\begin{align*}
I_{n,2}
=
&\frac{e^{-r'}}{2\pi} \int_{-\pi}^{\pi} \exp\left(g\left(e^{-e^{-r'}e^{i\varphi}}\right)+ (e^{-s^*}-1)\sum_{m= \floor{xn^*}}^n \frac{\theta_m}{m}e^{-me^{-r'}e^{i\varphi}}    +ne^{-r'}e^{i\varphi}+i\varphi      \right) \, d\varphi \\
=&
\frac{1}{e^{r'}2\pi} \int_{-\delta}^{\delta} \exp\bigl(f(\varphi) +i\varphi+O(e^{-r'}e^{i\varphi}) \big) \,d\varphi
\end{align*}
with
\begin{align*}
f(\varphi) 
&:= 
P(r'-i\varphi) +ne^{-r'}e^{i\varphi}+  (e^{-s^*}-1)\sum_{m= \floor{xn^*}}^n \frac{\theta_m}{m}e^{-me^{-r'}e^{i\varphi}} .
\end{align*}
We compute $I_{n,2}$ with the saddle-point method.
We start by splitting the integral $I_{n,2}$ into the regions $[-\delta,\delta]$ and $[-\pi/2, \pi/2]\setminus [-\delta,\delta]$ for some $\delta>0$ small determined below.
We denote by $I_{n,2,\delta}$ the integral over $[-\delta,\delta]$ and by $I_{n,2}^c:= I_{n,2} -I_{n,2,\delta}.$ 
We first consider the integral $I_{n,2,\delta}$. We have 
\begin{align}
 I_{n,2,\delta}
&=
\frac{1}{e^{r'}2\pi} \int_{-\delta}^{\delta} \exp\bigl(f(\varphi) +i\varphi+O(e^{-r'}e^{i\varphi}) \big) \,d\varphi
\end{align}
We begin by determining the behaviour of $f(\varphi)$ around $\varphi =0$ and thus write
\begin{align}
f(\varphi) 
&=
f(0) +i a(r') \varphi -  b(r') \frac{\varphi^2}{2} + R_n(\varphi,r').
\end{align}
% 
% -------------------------------
% 
In order to apply the saddle point method, we have to find $r'=r'(n,x)$ and $\delta =\delta(n,x)$ with 
$$b(r')\delta^2 \to\infty,\, \delta\to 0,\, a(r') =o\left( \sqrt{b(r')} \right) \text{ and }  R_n(\varphi,r') =o(\varphi^3 \delta^{-3}).$$
We now claim that we can choose 
\begin{align}
 r' = r + v
 \ \text{ with } \
 v:=\frac{(e^{-s^*}-1) r^k e^{-x}}{ne^{-r}}
 \ \text{ and } \
 \delta= (ne^{-r})^{-5/12}
 \label{eq:def_v_limit}
\end{align}
with $r$ as in \eqref{eq:thm:aux_asypmtotic_saddle_eq}. 
Recall, we have $P'(r) = ne^{-r}$ and $\overline{n}= r^k$. Further, we have see in \eqref{eq:saddle_solution_asympt}
\begin{align}
r = \log(n) - k\log\log(n) +O(1) \ \text{ as }n \to\infty.
\end{align}
Thus $v=O\big((\overline{n})^{-1/2}\big)=O(r^{-k/2})$.
Furthermore, we have for $\varphi\to 0$ and $\delta$ small 
\begin{align}
 e^{-me^{-r'}e^{i\varphi}}
=\, &
 e^{-me^{-r'}}
 -
 i m e^{-r'} e^{-me^{-r'}} \varphi
+
{e^{-m{e^{-r'}}}} \left(m{e^{-r'}}-\left(m {e^{-r'}} \right) ^{2} \right) \frac{\varphi^2}{2}\nonumber\\
&+
O\left(e^{-me^{-r'}/2} \left( me^{-r'} + (me^{-r'})^2+ (me^{-r'})^3
 \right) {\varphi}^{3}\right) .
\end{align}
This implies, together with Lemma~\ref{lem:eulermac1}, that
\begin{align}
b(r')
&=
P''(r') + ne^{-r'}+  
(e^{-s^*}-1)\sum_{m= \floor{xn^*}}^n \frac{\theta_m}{m} \left(me^{-r'}-(m e^{-r'})^{2}\right) e^{-me^{-r'}}\nonumber\\
&=
P''(r+v)+ ne^{-r-v}+ O( (e^{-s^*}-1)r^{k}) \nonumber\\
&=
P''(r) + ne^{-r}+ O(r^{k/2})
\sim 
r^{k}.
\label{eq:b(r)}
\end{align}
Thus $b(r')\delta^2 \to\infty$. We show as next that $a(r') =o(r^{k/2})$. We have
\begin{align}
a(r')
&=
-P'(r') + ne^{-r'}-  
(e^{-s^*}-1)e^{-r'}\sum_{m= \floor{xn^*}}^n \theta_m   e^{-me^{-r'}}\nonumber\\
&=
-P'(r+v) + ne^{-r-v} + (e^{-s^*}-1) r^k \int_{x}^{\infty} e^{-u}du + O(s^*r^{k-1})\nonumber\\
&=
-P'(r)+O(vr^{k-1}) + ne^{-r} - vne^{-r} + O(v^2ne^{-r}) - (e^{-s^*}-1) r^k  e^{-x} + O(s^*r^{k-1})\nonumber\\
&=
-vne^{-r} + (e^{-s^*}-1) r^k e^{-x}  + O(s^*r^{k-1}) 
=
O(s^*r^{k-1})
=
O(r^{k/2-1}),
\label{eq:a(r)}
\end{align}
where we have used on the last line the definition of $v$ in \eqref{eq:def_v_limit}.
We thus have indeed $a(r') =o\left( \sqrt{b(r')}\right)$.
A similar calculation shows that 
\begin{align}
  R_n(\varphi,r') =O(r^k\varphi^3) =o(\delta^{-3}\varphi^3).
\end{align}
Combining the above observations, we can use the same computation as in \eqref{eq:thm_aux_main_term_saddle} and obtain 
\begin{align*}
 I_{2,\delta}
 &=
 \frac{e^{-r'}}{2\pi}
 \int_{-\delta}^{\delta} \exp(f(\varphi)) \,d\varphi
 =
 \frac{\exp\big( f(0) \big) }{e^{r'} \sqrt{2\pi b(r') }} \left(1+ O\left(\frac{a(r')}{\sqrt{b(r')}}\right)  \right)\\
 &=
  \frac{\exp\big(-v+ f(0) \big) }{e^{r} \sqrt{2\pi(P''(r)+ ne^{-r})  }} \left(1+ O\left(r^{-1}\right)  \right).
\end{align*}
We now have to determine 
\begin{align*}
 f(0) 
 &= 
 P(r') +ne^{-r'}+  (e^{-s^*}-1)\sum_{m= \floor{xn^*}}^n \frac{\theta_m}{m}e^{-me^{-r'}} .
\end{align*}
We first look at $v$. We use  $ne^{-r} = P'(r)$ and obtain
\begin{align}
  \frac{(e^{-s^*}-1) r^k e^{-x}}{ne^{-r}} 
  &=
  e^{-x}(e^{-s^*}-1) \frac{ r^k }{P'(r)}\nonumber\\
  &=
   e^{-x}\left(- \frac{s}{\sqrt{\overline{n}}} + \frac{s^2}{2\overline{n}} +O\big( s^3(\overline{n})^{-3/2}\big) \right)\big(1+O(1/r) \big).
   \label{eq:v_as_fkt_in_s}
\end{align}
Inserting this and $r'=r+v$ gives
\begin{align}
 P(r')
 =&\,
 P(r+v)
 =
 P(r) + P'(r)v+ \frac{1}{2}P''(r) v^2 + O(v^3 \log^{k-2} n)\nonumber	\\
 =&\,
 P(r) - \frac{e^{-x}P'(r)(1+O(1/r))}{\sqrt{\overline{n}}}s\nonumber\\
 &+  \frac{(e^{-x}P'(r)+e^{-2x}P''(r))(1+O(1/r))}{2\overline{n}}s^2 + O\big( s^3(\overline{n})^{-1/2}\big)\nonumber\\
 =&\,
 P(r) - e^{-x}(1+O(1/r)) \sqrt{\overline{n}} s + e^{-x}(1+O(1/r))\frac{s^2}{2}+ O\big( s^3(\overline{n})^{-1/2}\big).
  \label{eq:limit_shape_P}
\end{align}
Furthermore, we have
\begin{align}
 ne^{-r'}
 =&\,
 ne^{-r-v}
 =
 ne^{-r}e^{-v}
 =
 P'(r)\big(1-v+v^2/2+O(v^3)\big)\nonumber\\
 =&\,
 P'(r) + e^{-x}\frac{ P'(r) (1+O(1/r))}{\sqrt{\overline{n}}}s + \frac{(e^{-2x}-e^{-x})P'(r)(1+O(1/r))}{2\overline{n}}s^2 + O\big( s^3(\overline{n})^{-1/2}\big) \nonumber\\
  =&\,
 ne^{-r}+ e^{-x}(1+O(1/r))\sqrt{\overline{n}}s + (e^{-2x}-e^{-x})(1+O(1/r))\frac{s^2}{2} + O\big( s^3(\overline{n})^{-1/2}\big)
\label{eq:limit_shape_ner}
\end{align}
and get with Lemma~\ref{lem:eulermac1}
\begin{align}
 (e^{-s^*}-1)\sum_{m= \floor{xn^*}}^n \frac{\theta_m}{m}e^{-ke^{-r'}} 
 &=
  (e^{-s^*}-1)\left(
    r^k \int_{x}^\infty  u^{-1} e^{-u}\,du+ O( r^{k-1}) \right) \nonumber \\
 &=
 w_\infty(x)\left( -\sqrt{\overline{n}}s +s^2/2 +O\big( s^3(\overline{n})^{-1/2}\big) \right)(1+O(1/r))
 \label{eq:limit_shape_ugly_sum}
\end{align}
where $ w_\infty(x) =  \int_{x}^\infty  u^{-1} e^{-u}\,du$ as in \eqref{eq:def_w_infinity}.
Combining \eqref{eq:limit_shape_P}, \eqref{eq:limit_shape_ner} and \eqref{eq:limit_shape_ugly_sum}, we obtain
\begin{align*}
 I_{2,\delta}
 =&\,
   \frac{\exp\big(P(r)+ ne^{-r}     \big) }{e^{r} \sqrt{2\pi(P''(r)+ ne^{-r})  }} \left(1+ O\left( r^{-1} \right)  \right)\nonumber\\
 &\times \exp\left(   -(w_\infty(x)+ z_n(x)) \sqrt{\overline{n}}s 
  +\left(e^{-2x}+w_\infty(x)\right) (1+O(1/r) \frac{s^2}{2}
  +O\big( s^3(\overline{n})^{-1/2}\big)  \right)  
\end{align*}
with $z_n(x) =O(1/r) =O(1/\log n)$.
Using Theorem~\ref{thm:aux_asypmtotic}, we immediately get that
 \begin{align*}
 \frac{I_{2,\delta}}{h_n}\cdot  \exp\left(w_\infty(x)(1+O(1/r)) \sqrt{\overline{n}}s \right)
 \longrightarrow &\,
 \exp\left(   \left(e^{-2x}+w_\infty(x)\right) \frac{s^2}{2} \right).  
\end{align*}
Comparing this with \eqref{eq:laplace_w_n_limit_shape},
we immediately see that $\frac{I_{2,\delta}}{h_n}$ has the behaviour of $\ET{\mathrm{exp}\bigl(-s\widetilde w_n(x)\bigr)}$, which is what we wanted to show.
Thus the proof is complete if we can show that the remaining integrals are of lower order.
 
We consider as next the integral $I_{2,\delta^c}$. We split this integral into the integrals over the intervals $[\delta, r^{-1/8}]$ and $[r^{-1/8},\pi/2]$.
We begin with the integral over the interval $[\delta, r^{-1/8}]$. For $|\varphi| \leq r^{-1/8}$ we have that 
\begin{align*}
  \Re\left( \sum_{m= \floor{xn^*}}^n \frac{\theta_m}{m}e^{-ke^{-r'}e^{i\varphi}}  \right)
  &=
   \sum_{m= \floor{xn^*}}^n \frac{\theta_m}{m}e^{-ke^{-r'}} -r^k\big(1+O(r^{-1})\big)\varphi^2 + O(r^k \varphi^4)\\
  &=
  r^k(w_{\infty}(x)+z_n(x))\big)  -r^k\big(1+O(r^{-1})\big)\varphi^2 + O(r^k \varphi^4).
\end{align*}
This implies
\begin{align*}
  \Re\left( (e^{-s^*}-1) \sum_{m= \floor{xn^*}}^n \frac{\theta_m}{m}e^{-ke^{-r'}e^{i\varphi}}  \right)
  \leq
  -s\sqrt{\overline{n}}(w_{\infty}(x)+z_n(x)) + \frac{3}{2}r^{k/2}\varphi^2 + O(1).
\end{align*}
Using the estimates in \eqref{eq:bound_P_and_ner}, we obtain as in \eqref{eq:estimate_I2c} and  \eqref{eq:estimate_I2c2} that
\begin{align*}
&\left|
\frac{1}{2\pi e^{r}} \int_{\delta}^{r^{-1/8}} \exp\bigl(f(\varphi)      \bigr) \, d\varphi \right|
\ll
e^{-r} \int_{\delta}^{r^{-1/8}}  \exp\left(\Re(f(\varphi))  \right) \, d\varphi 
\\
\ll& 
\exp\left(vP(r)+ ne^{-r} -s\sqrt{\overline{n}}(w_{\infty}(x)+z_n(x)) \right)e^{-r}\int_{\delta}^{r^{-1/8}}  \exp\left(-\frac{kvP(r)r^{-2}+ne^{-r}}{24}\varphi^2 \right) \, d\varphi\\
\ll& 
\frac{\exp\left(vP(r)+ ne^{-r} -s\sqrt{\overline{n}}(w_{\infty}(x)+z_n(x)) \right)}{e^{r}\delta\sqrt{ne^{-r}}}  e^{-\delta\sqrt{ne^{-r}}}.
\end{align*}
Thus this part of $I_{2,\delta^c}$ is indeed of lower order.
For the interval $[r^{-1/8},\pi/2]$, we use that 
\begin{align}
 \Re\left( (e^{-s^*}-1)\sum_{m= \floor{xn^*}}^n \frac{\theta_m}{m}e^{-ke^{-r'}e^{i\varphi}}  \right)
 =
  O(r^{k/2+1}) .
  \label{eq:crude_upper_bound_w}
\end{align}
Using again the same argument as in \eqref{eq:estimate_I2c} and  \eqref{eq:estimate_I2c2}, we obtain
\begin{align*}
&\left|
\frac{1}{2\pi e^{r}} \int_{r^{-1/8}}^{\pi} \exp\bigl(f(\varphi)      \bigr) \, d\varphi \right|
\ll
e^{-r} \int_{r^{-1/8}}^{\pi}  \exp\left(\Re(f(\varphi))  \right) \, d\varphi 
\\
\ll& 
\exp\left(vP(r)+ ne^{-r} + O(r^{k/2+1})  \right)e^{-r}\int_{r^{-1/8}}^{\pi}  \exp\left(-\frac{kvP(r)r^{-2}+ne^{-r}}{24}\varphi^2 \right) \, d\varphi\\
\ll& 
\frac{\exp\left(vP(r)+ ne^{-r} + O(r^{k/2+1})\right)}{e^{r}\sqrt{ne^{-r}}} 
\int_{r^{-1/8}\sqrt{ne^{-r}}}^{\infty}  \exp\left(-\frac{x^2}{2} \right) \, dx\\
\ll& 
\frac{\exp\left(vP(r)+ ne^{-r} + O(r^{k/2+1}) \right)}{e^{r}\delta\sqrt{ne^{-r}}}  e^{-r^{-1/4} ne^{-r}}.
\end{align*}
We now have $r^{-1/4} ne^{-r} \sim r^{k-1/4} > r^{k/2+1}$ since $k\geq 3$. This implies that this part of $I_{2,\delta^c}$ is  also of lower order.
Note that this inequality is the origin of the assumption $k\geq 3$ in this section.
It remains to consider the integral $I_{3}$. Here we use also the bound \eqref{eq:crude_upper_bound_w} and the fact that $k\geq 3$.
The computations closely parallel those of  the proof of Theorem~\ref{thm:aux_asypmtotic} and we may thus safely omit them.
This completes the proof. 
\end{proof}

\subsection{Proof of Theorem~\ref{thm:beh_increments_saddle}}
\label{sec:proof_increments}

The proof of Theorem~\ref{thm:beh_increments_saddle} has the same ingredients as the proof of Theorem~\ref{thm:limit_shape_saddle}.
We thus give only a sketch of the proof and highlight the necessary adjustments.

As for Theorem~\ref{thm:limit_shape_saddle}, we compute the Laplace transform of $\textbf{w}_n(\textbf{x})$.
We begin with the generating function. We have
\begin{lemma}[{\cite[Lemma~4.2]{CiZe13}}]
\label{lem:generating_w_n_x_y}
We have for $\textbf{x}=(x_1,\dots,x_\ell)\in\R^\ell$  with $x_\ell\geq x_{\ell-1}\geq \dots\geq x_1\geq 0$ and $\textbf{s}=(s_1,\dots,s_\ell)\in\C^\ell$ 
\begin{align}
\label{eq:generating_w_n_x_y}
\ET{\mathrm{exp}\bigl(-\langle \textbf{s},\textbf{w}_n(\textbf{x})\rangle \bigr)}
= 
\frac{1}{h_n}
[t^n] \left[\mathrm{exp}\left( g_\Theta(t) +\sum_{j=1}^\ell (e^{-s_j}-1)\sum_{k=\floor{x_j} }^{\floor{x_{j+1}-1} } \frac{\theta_k}{k}t^k \right) \right],
\end{align}
using the convention $x_{\ell+1}:=\infty$ and 
$\langle \textbf{s},\textbf{w}_n(\textbf{x})\rangle$ the standard scalar product of $\textbf{w}_n(\textbf{x})$ and $\textbf{s}$.

\end{lemma}
The first step is again to apply Cauchy's integral formula to \eqref{eq:generating_w_n_x_y} and 
to replace for all $j$ with $1\leq j\leq \ell$ the points $x_j$ by $x_j n^*$ and all $s_j$ by $s_j^*:=s_j(\overline{n})^{1/2}$. 
Further, we use the same curve as in the proof Theorem~\ref{thm:limit_shape_saddle}, but with a slightly different $r'$. 
Explicitly, we replace $r'$ by
\begin{align}
 r'_\ell = r + v_\ell
 \ \text{ with } \
 v_\ell:=\frac{\sum_{j=1}^\ell (e^{-s_j^*}-1) r^k (e^{-x_j}- e^{-x_{j+1}})}{ne^{-r}}
 \label{eq:def_v_limit2}
\end{align}
and use the same  $\delta= (ne^{-r})^{-5/12}$.
We then proceed to apply the saddle point method so that we arrive at
\begin{align*}
 \ET{\mathrm{exp}\bigl(-\langle \textbf{s},\widetilde{\textbf{w}}_n(\textbf{x})\rangle \bigr)}
 &=
  \frac{\exp\big(-v_\ell+ f_\ell(0) \big) }{e^{r} \sqrt{2\pi(P''(r)+ ne^{-r})  }} \left(1+ O\left(r^{-1}\right)  \right)
\end{align*}
with
\begin{align*}
f_\ell(\varphi) 
&:= 
P(r'-i\varphi) +ne^{-r'}e^{i\varphi}+  \sum_{j=1}^\ell (e^{-s_j^*}-1)  \sum_{k=\floor{x_jn^*} }^{\floor{x_{j+1}n^*-1} } \frac{\theta_m}{m} e^{-me^{-r'}e^{i\varphi}} .
\end{align*}
To prove the theorem, we have only to determine the coefficients of $s_j^2$ and $s_is_j$ in $f_\ell(0)$.
To do this, we first look at $v_\ell$. We use  $ne^{-r} = P'(r)$ and obtain
\begin{align}
  v_\ell
  &=
  \sum_{j=1}^\ell (e^{-x_j}- e^{-x_{j+1}})\left(- \frac{s_j}{\sqrt{\overline{n}}} + \frac{s_j^2}{2\overline{n}} +O\big( s_j^3(\overline{n})^{-3/2}\big) \right)\big(1+O(1/r) \big).
   \label{eq:v_as_fkt_in_s2}
\end{align}
Using the expansion
\begin{align*}
 P(r'_\ell)
 =&\,
 P(r+v_\ell)
 =
 P(r) + P'(r)v_\ell+ \frac{1}{2}P''(r) v_\ell^2 + O(v_\ell^3 \log^{k-2} n),
\end{align*}
and $P'(r) \sim r^k$ and $P''(r) =O(r^{k-1})$,
we immediately get 
\begin{align}
 [s_j^2]&\left[P(r'_\ell)\right]
 =\,
 \frac{(e^{-x_j}- e^{-x_{j+1}})}{2} (1+O(1/r)) \ &&\text{ for } 1\leq j\leq \ell,\\
  [s_is_j]&\left[P(r'_\ell)\right]
 =\,
 O(1/r))  
 \ &&\text{ for } i\neq j.
  \label{eq:limit_shape_P3}
\end{align}
Furthermore, using
\begin{align*}
 ne^{-r'_\ell}
 =&\,
 ne^{-r-v_\ell}
 =
 ne^{-r}e^{-v_\ell}
 =
 P'(r)\big(1-v_\ell+v_\ell^2/2+O(v_\ell^3)\big),
\end{align*}
we obtain
\begin{align}
 [s_j^2]&\left[ne^{r'_\ell}\right]
 =\,
 \frac{ (e^{-x_j}- e^{-x_{j+1}})^2   - (e^{-x_j}- e^{-x_{j+1}})  }{2}(1+O(1/r))
 \ &&\text{ for } 1\leq j\leq \ell,\\
  [s_is_j]&\left[ne^{r'_\ell}\right]
 =\,
(e^{-x_j}- e^{-x_{j+1}})(e^{-x_i}- e^{-x_{i+1}})(1+O(1/r))
 \ &&\text{ for } i\neq j.
  \label{eq:limit_shape_ner2}
\end{align}
Finally, applying Lemma~\ref{lem:eulermac1}, we get 
\begin{align}
 \sum_{j=1}^\ell (e^{-s_j^*}-1)  \sum_{k=\floor{x_jn^*} }^{\floor{x_{j+1}n^*-1} } \frac{\theta_m}{m} e^{-me^{-r'}} 
 &=
  \sum_{j=1}^\ell (e^{-s_j^*}-1)  \left(
    r^k \int_{x_j}^{x_{j+1}}  u^{-1} e^{-u}\,du+ O( r^{k-1}) \right).
\end{align}
This implies
\begin{align}
  [s_j^2]&\left[ \frac{1}{2}\sum_{j=1}^\ell (e^{-s_j^*}-1)  \sum_{k=\floor{x_jn^*} }^{\floor{x_{j+1}n^*-1} } \frac{\theta_m}{m} e^{-me^{-r'}} \right]
   =\,
  \int_{x_j}^{x_{j+1}}  u^{-1} e^{-u}\,du(1+O(1/r)),\\
  [s_is_j]&\left[ \sum_{j=1}^\ell (e^{-s_j^*}-1)  \sum_{k=\floor{x_jn^*} }^{\floor{x_{j+1}n^*-1} } \frac{\theta_m}{m} e^{-me^{-r'}} \right]
   =\,
   O(1/r).
    \label{eq:limit_shape_ugly_sum2}
\end{align}
Combining all these equations, we obtain
\begin{align}
  [s_j^2]&\left[ f_\ell(0) \right]
   =\,
   \frac{ (e^{-x_j}- e^{-x_{j+1}})^2 +\int_{x_j}^{x_{j+1}}  u^{-1} e^{-u}\,du(1+O(1/r))  }{2}(1+O(1/r)),\\
  [s_is_j]&\left[ f_\ell(0)\right]
   =\,
  (e^{-x_j}- e^{-x_{j+1}})(e^{-x_i}- e^{-x_{i+1}})(1+O(1/r)).
    \label{eq:limit_shape_allsums2}
\end{align}
This completes the proof Theorem~\ref{thm:beh_increments_saddle}.

% -----------------------------------------------------------
% 
% 
 
\subsection{Proof of Theorem~\ref{thm:func_CLT_saddle}}
\label{sec:proof_func_CLT}
 
We use here the same method of proof as in \cite[Section~4.3]{CiZe13} and as in \cite{Ha90}. 
Theorem~\ref{thm:beh_increments_saddle} gives us the convergence of the finite dimensional distributions.
It thus remains to prove the tightness of the process. 
This means we have to show that the moment condition in \cite[p.128]{Bi99} is fulfilled.
We begin with the generating function. We have
\begin{lemma}[{\cite[Lemma~4.10]{CiZe13}}]
\label{lemma:tightness_saddle_generating}
For $0\leq x_1<x\leq x_2$ arbitrary and $x^*:={x n^*}$, $x_1^*:={x_1 n^*}$ and $x_2^*:={x_2 n^*}$
\begin{align}
\label{eq:lem_tightness_saddle_generating}
&(\overline n)^2\cdot h_n\,\ET{\big(\widetilde w_n(x^*)-\widetilde w_n(x_1^*)\big)^2 \big(\widetilde w_n(x_2^*)-\widetilde w_n(x^*)\big)^2}\\
&=
[t^n]\left[\left((g_{x_1^*}^{x^*}(t)-E_{x_1}^{x})^2 + g_{x_1^*}^{x^*}(t) \right) \left((g_{x^*}^{x_2^*}(t)-E_{x}^{x_2})^2 + g_{x^*}^{x_2^*}(t) \right) \mathrm{exp}(g_\Theta(t)) \right]\nonumber
\end{align}
with
\begin{align*}
 g_{a}^b(z):=\sum_{a\leq j <b} \frac{\vartheta_j}{j}z^j 
 \ \text{ and }\
 E_a^b = \ET{ w_n(bn^*)- w_n(an^*)}
 \ \text{ for }a<b.
\end{align*}
\end{lemma}
We can now prove the tightness of the process $\widetilde w_n(x^*)$.
\begin{lemma}
\label{lemma:tightness_saddle} 
We have for $0\leq x_1<x\leq x_2<K$ with $K$ arbitrary 
\begin{equation}\label{eq:goal_2}
\ET{(\widetilde w_n(x^*)-\widetilde w_n(x_1^*))^2 (\widetilde w_n(x_2^*)-\widetilde w_n(x^*))^2}
=
O\big((x_2-x_1)^2\big).
\end{equation}
\end{lemma}
\begin{proof}
We use Lemma~\ref{lemma:tightness_saddle_generating} and apply the proof of Theorem~\ref{thm:aux_asypmtotic} to the function
$$
g_n(t)
:=
\exp\left(
g_\Theta(t)
+
\log\left((g_{x_1^*}^{x^*}(t)-E_{x_1}^{x})^2 + g_{x_1^*}^{x^*}(t) \right)
+
\log\left((g_{x^*}^{x_2^*}(t)-E_{x_1}^{x})^2 + g_{x^*}^{x_2^*}(t) \right)
\right).
$$ 
We claim that we can use the same curve and the same $r$ and $\delta$ as in the proof of Theorem~\ref{thm:aux_asypmtotic}.
Theorem~\ref{thm:limit_shape_saddle} and Lemma~\ref{lem:eulermac1} imply immediately that
$E_{x_1}^{x} = O(\overline{n})$ and $g_{x_1^*}^{x^*}(e^{e^{-r}e^{i\varphi}  })= O(\overline{n})$.
It is thus immediate to show that we indeed can use the same curve and the same $r$ and $\delta$.
We thus arrive at
\begin{align*}
&(\overline n)^2\,\ET{\big(\widetilde w_n(x^*)-\widetilde w_n(x_1^*)\big)^2 \big(\widetilde w_n(x_2^*)-\widetilde w_n(x^*)\big)^2}\\
&=
\left((g_{x_1^*}^{x^*}(e^{-e^{-r}})-E_{x_1}^{x})^2 + g_{x_1^*}^{x^*}(e^{-e^{-r}}) \right) \left((g_{x^*}^{x_2^*}(e^{-e^{-r}})-E_{x}^{x_2})^2 + g_{x^*}^{x_2^*}(e^{-e^{-r}}) \right) \big(1+o(1)\big).
\end{align*}
Differentiating \eqref{eq:generating_w_n_x_y} with respect to $s_1$ and substituting $s_1=0$ shows that
 \begin{align}
   E_{x_1}^x = \ET{\widetilde w_n(x^*)-\widetilde w_n(x_1^*)}
   =
   \frac{1}{h_n}[t^n]\left[ g_{x^*}^{x_2^*}(t) \exp(g_\Theta(t))\right]
   =
   g_{x^*}^{x_2^*}(e^{-e^{-r}})(1+o(1)).
 \end{align}
It is then clear that $g_{x_1^*}^{x^*}(e^{-e^{-r}})-E_{x_1}^x =o(x-x_1)$. 
Therefore 
$$
\left((g_{x_1^*}^{x^*}(e^{-e^{-r}})-E_{x_1}^{x})^2 + g_{x_1^*}^{x^*}(e^{-e^{-r}})\right)(\overline n)^{-1}=O\left(g_{x_1^*}^{x^*}(e^{-e^{-r}})(\overline n)^{-1}\right).
$$
Applying Lemma~\ref{lem:eulermac1} then shows  $g_{x_1^*}^{x^*}(e^{-e^{-r}})(\overline n)^{-1}=O(x-x_1)$. Similar considerations apply for $x_2$. This completes the proof.
\end{proof}

\section{Total variation distance}
\label{sec:dtv}
We have proven in Section~\ref{sec:cyclecounts} that for each $b\in \N$
\begin{align}\label{eq:intro_convergence_fixed_b}
(C_1^n, C_2^n, \ldots, C_b^n) \overset{d}{\longrightarrow} (Y_1, Y_2, \ldots,Y_b), \quad \quad \text{as } n \rightarrow \infty.
\end{align}
with $(Y_m)_{m=1}^b$  mutually independent Poisson random variables with $\E{Y_{m}}=\frac{\theta_m}{m}$ for all $m$.
Unfortunately, the convergence in \eqref{eq:intro_convergence_fixed_b} is often not strong enough, since many interesting random variables involve all or almost all cycle counts $C_m$.
Thus, one needs estimates where $b$ and $n$ grow simultaneously.
The quality of the approximation can conveniently be described in terms of the total variation distance. For all $1 \leq b \leq n$ denote by $d_b(n)$ the total variation distance
\begin{align}\label{eq:intro_tot_var_dist_d_bn}
d_b(n) 
&:= 
d_{\operatorname{TV}}\big(\mathcal{L}(C_1^n, C_2^n, ..., C_b^n), \mathcal{L}(Y_1, Y_2, ..., Y_b)\big)
\end{align}

The main result of this section is

\begin{theorem}
\label{thm:DTV}
Let $(b(n))_{n\in\N}$ be a sequence so that $b(n) = o \big( n^c\big)$ with $0<c< (3k+3)^{-\frac{1}{k+1}}$. Then one has that
\begin{align}
 d_b(n) = o\big(1  \big).
\end{align}
\end{theorem}
\begin{remark}
The computations in the proof of Theorem~\ref{thm:DTV} and the similarities with the cases $\theta_m\approx \vartheta$ and $\theta_m\sim m^\gamma$ strongly suggest that
Theorem~\ref{thm:DTV} might not be optimal. We expect that $ d_b(n) = o\big(1  \big)$ if and only if $b(n) = o(n/\log^k(n))$. 
However, our current estimates for the error terms are too weak to prove this and a more sophisticated bound would be needed.  
\end{remark}

For the proof of Theorem~\ref{thm:DTV}, we follow the ideas in \cite{ArTa92c}.
These ideas have been for instance successfully applied in \cite{StZe14a} for the case $\theta_m \sim m^\alpha$ and in \cite{BeScZe17} for random permutations without macroscopic cycles.
Before, we can prove Theorem~\ref{thm:DTV}, we have to make some preparations and introduce some notations. 

Let $\left(Y_{m}\right)_{m\in\N}$ be independent random variables with $Y_{m}\sim\mathrm{Poi}\left(\frac{\theta_m}{m}\right)$ for all $m\in\mathbb{N}$.
We use the notation $\boldsymbol{Y}_b :=\left(Y_1, Y_2, \dots, Y_{b(n)} \right)$ and $\boldsymbol{C}_b:=\left(C_1,C_2,\dots, C_{b(n)}\right)$ for the vector of the cycle counts up to length ${b(n)}$, 
and $\boldsymbol{a}=\left(a_1,a_2,\dots, a_{b(n)}\right)$ for a vector $\boldsymbol{a}\in\N^{b(n)}$. 
Inserting the definition of the total variation distance, we get
\begin{align}
d_b(n) 
=
\frac{1}{2}
\sum_{\boldsymbol{a}\in\N^{b(n)}} |\PT{\boldsymbol{C}_{b}=\boldsymbol{a}} - \Pb{\boldsymbol{Y}_{b}=\boldsymbol{a}}|.
\label{eq:def_dtv}
\end{align}

A corner stone for investigating the classical case of uniform random permutation in \cite{ArTa92c} is the so-called conditioning relation. % \cite[Equation (1.15)]{ABT02}.
To formulate this, 
let us define
\begin{equation}
T_{b_{1}b_{2}}:=\sum_{k=b_{1}+1}^{b_{2}}kY_{k}
\ \text{ for } \
b_1,b_2\in\N
\ \text{ with } \ b_1\leq b_2.
\label{eq:Tdef}
\end{equation}
The conditioning relation \cite[Equation (1.15)]{ABT02} now states that 
\begin{equation}
\PT{\boldsymbol{C}_{b}=\boldsymbol{a}}
=
\mathbb{P}\left[\left.\boldsymbol{Y}_{b}=\boldsymbol{a}\right|T_{0n}=n\right].
\label{eq:Poissondarstellung}
\end{equation}
It is direct to see that \eqref{eq:Poissondarstellung} indeed holds also for $\mathbb{P}_\Theta$.
Inserting \eqref{eq:Poissondarstellung} in \eqref{eq:def_dtv} and using the same computation as in the proof of \cite[Lemma~3.1]{ABT02}, one immediately obtains
\begin{align}
d_b(n) 
=
\sum_{\ell=0}^{\infty} 
\mathbb{P}\left[T_{0b\left(n\right)}=\ell\right] 
\left(1-\frac{\mathbb{P}\left[T_{b\left(n\right)n}=n-\ell\right]}{\mathbb{P}\left[T_{0n}=n\right]}\right)_+
\label{eq:dtv_with_B}
\end{align}
with $(x)_+=\max\{x,0\}$. Using this, we now can prove Proof of Theorem~\ref{thm:DTV}.
\begin{proof}[Proof of Theorem~\ref{thm:DTV}]
We spilt the sum in \eqref{eq:dtv_with_B} into a central and a non-cenral part.
Explicitly, we set
\begin{align}
 J:=\left[\E{T_{0b\left(n\right)}} - g(n)b(n)\log^{k/2} b(n), \, \E{T_{0b\left(n\right)}} + g(n)b(n)\log^{k/2} b(n)   \right] 
 \label{eq:def_intervall_dtv}
\end{align}
for some $g(n)$ with $g(n)\to\infty$ and $g(n)=o\big( \log^{k/2}(b(n))\big)$. We thus obtain
\begin{align}
d_b(n) 
\leq 
\Pb{T_{0b(n)}\notin J}
+\max_{\ell\in J} 
\left(1-\frac{\Pb{T_{b(n)n}=n-\ell}}{\Pb{T_{0n}=n}}\right)_+.
\label{eq:trivial_upper_bound_dtv}
\end{align}
We first look at the summand $\Pb{T_{0b(n)}\notin J}$. 
Recall, we have seen in \eqref{eq:alternative_cycle_weights2} that
\begin{align}
 \theta_m = \log^k(m) +\sum_{j=0}^{k-1} a_j \log^j(m) + o(1) \ \text{ for some } a_j\in\R.
 \label{eq:alternative_cycle_weights2dtv}
\end{align}
Using this and the fact that $\left(Y_{m}\right)_{m\in\N}$ are independent Poisson random variables, we get
\begin{align}
\E{T_{0b(n)}} 
&=
\sum_{m=1}^{b(n)} m\E{Y_m} 
= 
\sum_{m=1}^{b(n)} \theta_m
\sim 
b(n)\log^{k}(b(n)),\\
\mathrm{Var}(T_{0b(n)})
&=
\sum_{m=1}^{b(n)} m^2 \mathrm{Var}(Y_m)
=
\sum_{m=1}^{b(n)} m \theta_m \sim b^2(n) \log^{k} b(n).
\end{align}
Thus Chebyshev's inequality implies
\begin{align}
 \Pb{T_{0b\left(n\right)}\notin J}
 \leq \frac{\mathrm{Var}(T_{0b\left(n\right)})}{\big(g(n)b(n)\log^{k/2} b(n) \big)^2}
 =
 O\big( g^{-2}(n)\big).
\end{align}
This shows that $ \Pb{T_{0b\left(n\right)}\notin J}$ is $o(1)$. 
It thus remains to show that the second summand in \eqref{eq:trivial_upper_bound_dtv} is also $o(1)$.
Note that the probability generating function of $T_{b_1b_2}$ is given by
\begin{align}
 \E{z^{T_{b_1b_2}}}
 =
 \exp\left( \sum_{m=b_1+1}^{b_2} \frac{\theta_m}{m} \left( z^m-1\right) \right).
\end{align}
We thus obtain
\begin{align}
 \Pb{T_{b(n)n}=n-\ell} 
 &= 
%  e^{ -\sum_{m=b(n)+1}^{n} \frac{\theta_m}{m}}
  \exp\left( -\sum_{m=b(n)+1}^{n} \frac{\theta_m}{m}\right)
 [z^{n-\ell}]
 \left[
 \exp\left(  \sum_{m=b(n)+1}^{n} \frac{\theta_m}{m}z^m \right)
  \right]\nonumber\\
 &=
  \exp\left( -\sum_{m=b(n)+1}^{n} \frac{\theta_m}{m}\right)
 [z^{n-\ell}]
 \left[
 \exp\left(g_\Theta(z) - \sum_{m=1}^{b(n)} \frac{\theta_m}{m}z^m \right)
  \right],
 \\
 \Pb{T_{0n}=n}
  &= 
 \exp\left( -\sum_{m=1}^{n} \frac{\theta_m}{m}\right)
  [z^n]
  \left[
 \exp\left(  \sum_{m=1}^{n} \frac{\theta_m}{m}z^m \right)
 \right]\nonumber\\
 &=
  \exp\left( -\sum_{m=1}^{n} \frac{\theta_m}{m}\right)
  [z^n]
  \left[
 \exp\left(  g_\Theta(z) \right)
 \right].
\end{align}
Using Corollary~\ref{cor:Hn_generatingN}, we immediately get 
\begin{align}
 \frac{\Pb{T_{b(n)n}=n-\ell}}{\Pb{T_{0n}=n}}
 =
  \frac{\exp\left( \sum_{m=1}^{b(n)} \frac{\theta_m}{m}\right)}{h_n}
 [z^{n-\ell}]
 \left[
 \exp\left(g_\Theta(z) - \sum_{m=1}^{b(n)} \frac{\theta_m}{m}z^m \right)
  \right].
  \label{eq:dtv_sh9}
\end{align}
Theorem~\ref{thm:aux_asypmtotic} gives us the asymptotic behaviour of $h_n$.
Thus it remains to compute
\begin{align}
 I^{tv}_{n}:=
 [z^{n-\ell}]
 \left[
 \exp\left(g_\Theta(z) - \sum_{m=1}^{b(n)} \frac{\theta_m}{m}z^m \right)
  \right].
   \label{eq:for_dtv_17}
\end{align}
We do this similar as for the proof of Theorem~\ref{thm:aux_asypmtotic}.
Cauchy's integral formula and the change of variable $z=e^{-w}$ gives us 
\begin{align}
I^{tv}_{n}
=
  \frac{1}{2\pi i}
 \int_{\gamma''}
  \exp\left( 
  P(-\log(w))+ (n-\ell) w 
  - 
  \sum_{m=1}^{b(n)} \frac{\theta_m}{m} e^{-mw}+  O(w)    
  \right) \, dw,
\end{align}
where $\gamma''$ is as in the proof of Theorem~\ref{thm:aux_asypmtotic}.
Thus we have
$\gamma'' = \gamma''_1\cup \gamma''_2\cup \gamma''_3$ with 
 \begin{align*}
\gamma''_1&:=\{w=(-\pi+x)i,x\in[0,\pi-e^{-r}]\},\\
\gamma''_2&:=\{w=e^{-r}e^{i\varphi},\varphi\in[-\pi/2 ,\pi/2\},\\
\gamma''_3&:=\{w=ix ,x\in[e^{-r},\pi]\}
 \end{align*}
and $r$ the solution of the equation $ne^{-r}= P'(r)$.
We now split the curve $\gamma''$ into two parts. 
Explicitly, we denote by $\gamma''^2$ the part of $\gamma''$ consisting of 
all $w$ with $w$ with $|w|\leq \frac{1}{b(n)\log^{2k}(b(n))}$
and by $\gamma''^{1,3}$ the remaining parts of $\gamma''$.

We begin by computing the integral over $\gamma''^2$.
For this, we need two observations.
First, we have $w\ell =o(1)$ for all $\ell \in J$ and all $w$ in $\gamma''^2$.
Further, using \eqref{eq:alternative_cycle_weights2dtv}, we get for all $w$ in $\gamma''^2$
\begin{align}
 \sum_{m=1}^{b(n)} \frac{\theta_m}{m}\exp(-mw)
 &=
 \sum_{m=1}^{b(n)} \frac{\theta_m}{m} \left(1+O\left(mw \right)\right)
 =
 \sum_{m=1}^{b(n)} \frac{\theta_m}{m}  + O\left(w  \sum_{m=1}^{b(n)} \theta_m\right)\nonumber\\
 &=
  \sum_{m=1}^{b(n)} \frac{\theta_m}{m}  + O\big(w b(n)\log^{k}(b(n))\big) %\nonumber\\
  =
  \sum_{m=1}^{b(n)} \frac{\theta_m}{m}  + o\big( 1\big).
  \label{eq:be_sum_1}
\end{align}
Inserting this into the integral over $\gamma''^2$, we obtain
\begin{align*}
I^{2}_{n}
  &:=
  \frac{1}{2\pi i}
 \int_{\gamma''^2}
  \exp\left( 
%   f_{b(n),\ell}(\varphi)
  P(-\log(w))+ (n-\ell) w 
  - 
  \sum_{m=1}^{b(n)} \frac{\theta_m}{m} e^{-mw}+  O(w)    
  \right) \, dw\\
   &=
     \frac{\exp\left( -\sum_{m=1}^{b(n)} \frac{\theta_m}{m}  +o(1)\right)}{2\pi i}
 \int_{\gamma''^2}
  \exp\left( 
%   f_{b(n),\ell}(\varphi)
  P(-\log(w))+ nw +  O(w)    
  \right) \, dw.
\end{align*}
This is now the same integral as in the proof of Theorem~\ref{thm:aux_asypmtotic}. Thus we get 
\begin{align}
I^{2}_{n}
  &=
  \exp\left( -\sum_{m=1}^{b(n)} \frac{\theta_m}{m}  +o(1)\right)
  \cdot
\frac{\exp\left(P(r)+ ne^{-r}\right) }{e^{r}\sqrt{2\pi P''(r)+2\pi ne^{-r}}} \left(1+ O(\log^{-k/2}(n)) \right)\nonumber\\
  &=
  \exp\left( -\sum_{m=1}^{b(n)} \frac{\theta_m}{m}  +o(1)\right)
  \cdot
  h_n\left(1+ O(\log^{-k/2}(n)) \right).
\label{eq:I_tv2}
\end{align}
Inserting \eqref{eq:I_tv2} into \eqref{eq:dtv_sh9}, we obtain
\begin{align}
 \frac{\Pb{T_{b(n)n}=n-\ell}}{\Pb{T_{0n}=n}}
 =
 1+ o(1) 
 + 
  \frac{\exp\left( \sum_{m=1}^{b(n)} \frac{\theta_m}{m}\right)}{h_n} I^{1,3}_{n},
\end{align}
where $I^{1,3}_{n}$ denotes the integral over $\gamma''^{1,3}$.
It thus remains to show that $\frac{\exp\left( \sum_{m=1}^{b(n)} \frac{\theta_m}{m}\right)}{h_n} I^{1,3}_{n}$ is also $o(1)$.
Clearly, $\gamma''^{1,3}$ is a part of $\gamma''_1$ and $\gamma''_3$. 
Thus we obtain as in the proof of Theorem~\ref{thm:aux_asypmtotic} that
\begin{align*}
|I^{1,3}_{n}| 
&=
 \left|\frac{1}{2\pi i}
 \int_{\gamma''^{1,3}}
  \exp\left( 
%   f_{b(n),\ell}(\varphi)
  P(-\log(w))+ (n-\ell) w 
  - 
  \sum_{m=1}^{b(n)} \frac{\theta_m}{m} e^{-mw}+  O(w)    
  \right) \, dw\right|\\
&\leq
 \frac{2}{2\pi}\int_{ \frac{1}{b(n)\log^{2k}(b(n))}}^\pi 
 	\exp\left( \Re\big(P(-\log(x)-i\pi/2)\big)  - 
 	  \Re\left(\sum_{m=1}^{b(n)} \frac{\theta_m}{m} e^{-mix}\right)\right) \, dx.
\end{align*}
We have by assumption $b(n)=O(n^c)$ with $c< (3k+3)^{-\frac{1}{k+1}}$. This implies
\begin{align*}
 \Re\left( -\sum_{m=1}^{b(n)} \frac{\theta_m}{m} e^{-imx} \right) 
 &\leq
 \sum_{m=1}^{b(n)} \frac{\theta_m}{m}  
 =
 \log^{k+1}b(n)+ O\left(\log^{k}b(n)\right )\\
 &\leq
 c^{k+1}\log^{k+1}n+ O\left(\log^{k}(n)\right ).
\end{align*}
Together with the inequality \eqref{eq:bound_gamma3} and Lemma~\ref{lem:int_exp_P}, we obtain for each $\epsilon>0$ 
\begin{align}
	|I^{1,3}_{n}|
	&\leq 
	 \frac{\exp\left(  c^{k+1}\log^{k+1}n+ O\left(\log^{k}(n)\right )  \right)}{\pi}
		\int_{n^{-c-\epsilon}}^\pi 
	 \exp\left(P(-\log(x)) - \frac{9}{8} P''(-\log(x)) \right) \, dx\nonumber\\
      &=
\frac{\exp\left( c^{k+1}\log^{k+1}n+ O\left(\log^{k}(n)\right )\right)}{\pi} 
\int_{e^{-\pi}}^{(c+\epsilon) \log n} \exp\Bigl(  P(y) -  \frac{9}{8} P''(y) \Bigr) e^{-y} \, dy\nonumber\\
&=
\frac{\exp\left( 2(c+\epsilon)^{k+1}\log^{k+1}n+ O\left(\log^{k}(n)\right )\right)}{\pi  P'((c+\epsilon) \log n) }.
\label{eq:I_tv3}
\end{align}
Since $c<  (3k+3)^{-\frac{1}{k+1}}$ and $P(r)= \frac{1}{k+1}r^{k+1} +O(r^k)$, we get for $\epsilon$ small enough
\begin{align}
 \left|I^{1,3}_{n} \exp\left( \sum_{m=1}^{b(n)} \frac{\theta_m}{m}\right)\right|
 &\leq
 \frac{\exp\left( 3(c+\epsilon)^{k+1}\log^{k+1}n+ O\left(\log^{k}(n)\right )\right)}{\pi  P'((c+\epsilon) \log n) }\\
 &\leq 
 \exp\left( (1-\epsilon')P(r) \right)
\end{align}
for some $\epsilon'>0$. Recall, we have
\begin{align*}
 h_n=\frac{\exp\left(P(r)+ ne^{-r}\right) }{e^{r}\sqrt{2\pi P''(r)+2\pi ne^{-r}}} \left(1+ O(\log^{-k/2}(n))\right).
\end{align*}
This implies that
\begin{align}
 \frac{\exp\left( \sum_{m=1}^{b(n)} \frac{\theta_m}{m}\right)}{h_n} I^{1,3}_{n}
 =
 O\left(\exp\left( -\frac{\epsilon'}{2}P(r) \right) \right) 
 =
 o(1).
\end{align}
This completes the proof of Theorem~\ref{thm:DTV}.
\end{proof}

\bibliography{literatur}

\begin{thebibliography}{10}

\bibitem{ABT02}
R.~Arratia, A.~Barbour, and S.~Tavar{\'e}.
\newblock {\em Logarithmic combinatorial structures: a probabilistic approach}.
\newblock EMS Monographs in Mathematics. European Mathematical Society (EMS),
  Z{\"u}rich, 2003.

\bibitem{ArTa92c}
R.~Arratia and S.~Tavar{\'e}.
\newblock The cycle structure of random permutations.
\newblock {\em Ann. Probab.}, 20(3):1567--1591, 1992.

\bibitem{BeScZe17}
V.~{Betz}, H.~{Sch{\"a}fer}, and D.~{Zeindler}.
\newblock {Random permutations without macroscopic cycles}.
\newblock Dec. 2017.

\bibitem{BeUe10}
V.~Betz and D.~Ueltschi.
\newblock Spatial random permutations and poisson-dirichlet law of cycle
  lengths.
\newblock {\em Electron. J. Probab.}, 16:no. 41, 1173--1192, 2011.

\bibitem{BeUeVe11}
V.~Betz, D.~Ueltschi, and Y.~Velenik.
\newblock Random permutations with cycle weights.
\newblock {\em Ann. Appl. Probab.}, 21(1):312--331, 2011.

\bibitem{Bi99}
P.~Billingsley.
\newblock {\em Convergence of probability measures}.
\newblock Wiley Series in Probability and Statistics: Probability and
  Statistics. John Wiley \& Sons Inc., New York, second edition, 1999.
\newblock A Wiley-Interscience Publication.

\bibitem{BoZe14}
L.~V. Bogachev and D.~Zeindler.
\newblock Asymptotic statistics of cycles in surrogate-spatial permutations.
\newblock {\em Communications in Mathematical Physics}, pages 1--78, 2014.

\bibitem{Ca07}
P.~Chareka.
\newblock A finite-interval uniqueness theorem for bilateral {L}aplace
  transforms.
\newblock {\em Int. J. Math. Math. Sci.}, pages Art. ID 60916, 6, 2007.

\bibitem{CiZe13}
A.~Cipriani and D.~Zeindler.
\newblock The limit shape of random permutations with polynomially growing
  cycle weights.
\newblock {\em ALEA Lat. Am. J. Probab. Math. Stat.}, 12(2):971--999, 2015.

\bibitem{dr01}
A.~Dunn and N.~Robles.
\newblock Polynomial partition asymptotics.
\newblock {\em Journal of Mathematical Analysis and Applications},
  459:359--384, 2018.

\bibitem{ErUe11}
N.~M. Ercolani and D.~Ueltschi.
\newblock Cycle structure of random permutations with cycle weights.
\newblock {\em Random Structures Algorithms}, 44(1):109--133, 2014.

\bibitem{Ew72}
W.~J. Ewens.
\newblock The sampling theory of selectively neutral alleles.
\newblock {\em Theoret. Population Biology}, 3:87--112; erratum, ibid. 3
  (1972), 240; erratum, ibid. 3 (1972), 376, 1972.

\bibitem{Fl99}
P.~Flajolet.
\newblock Singularity analysis and asymptotics of {B}ernoulli sums.
\newblock {\em Theoret. Comput. Sci.}, 215(1-2):371--381, 1999.

\bibitem{FlSe09}
P.~Flajolet and R.~Sedgewick.
\newblock {\em Analytic Combinatorics}.
\newblock Cambridge University Press, New York, NY, USA, 2009.

\bibitem{Ha90}
J.~C. Hansen.
\newblock A functional central limit theorem for the {E}wens sampling formula.
\newblock {\em J. Appl. Probab.}, 27(1):28--43, 1990.

\bibitem{Ho87}
F.~M. Hoppe.
\newblock The sampling theory of neutral alleles and an urn model in population
  genetics.
\newblock {\em J. Math. Biol.}, 25(2):123--159, 1987.

\bibitem{Ki77}
J.~F.~C. Kingman.
\newblock The population structure associated with the {E}wens sampling
  formula.
\newblock {\em Theoret. Population Biology}, 11(2):274--283, 1977.

\bibitem{Mac95}
I.~G. Macdonald.
\newblock {\em Symmetric functions and {H}all polynomials}.
\newblock Oxford Mathematical Monographs. The Clarendon Press Oxford University
  Press, New York, second edition, 1995.
\newblock With contributions by A. Zelevinsky, Oxford Science Publications.

\bibitem{NiZe11}
A.~Nikeghbali and D.~Zeindler.
\newblock The generalized weighted probability measure on the symmetric group
  and the asymptotic behaviour of the cycles.
\newblock {\em Annales de L'Institut Poincar\'e}, 49, no.4:961--981, 2011.

\bibitem{Po37}
G.~P{\'o}lya.
\newblock Kombinatorische anzahlbestimmungen f{\"u}r gruppen, graphen, und
  chemische verbindungen.
\newblock {\em Acta Mathematica}, 68:145--254, 1937.

\bibitem{ShVe77}
A.~Shmidt and A.~M. Vershik.
\newblock Limit measures arising in the asymptotic theory of symmetric groups.
\newblock {\em Theory Probab. Appl.}, 22, No.1:70--85, 1977.

\bibitem{StZe14a}
J.~Storm and D.~Zeindler.
\newblock Total variation distance and the {E}rd\"os-{T}ur\'an law for random
  permutations with polynomially growing cycle weights.
\newblock {\em Ann. Inst. Henri Poincar\'e Probab. Stat.}, 52(4):1614--1640,
  2016.

\end{thebibliography}
\bibliographystyle{abbrv}

\end{document}